%% file: main.tex
\documentclass{article}
\usepackage{graphicx, amsfonts, amsmath,amssymb, amsthm}
\usepackage{xcolor,longtable,array,colortbl}

\usepackage{etoolbox}
\usepackage{booktabs}
\usepackage{tikz}
\usepackage{subcaption}
\usepackage{float}
\usepackage[backend=biber]{biblatex}
\bibliography{sources.bib}

\title{Genus zero surfaces with volume and principal curvature less than a unit sphere}
\date{}
\author{Matthew Bolan}

\newtheorem{theorem}{Theorem}
\newtheorem{lemma}{Lemma}
\newtheorem{remark}{Remark}

\begin{document}

\maketitle

\section{Introduction}
The Pestov-Ionin theorem \cite{pestovionin} states that every simple closed $C^2$ curve in $\mathbb R^2$ having curvature in $[-1,1]$ encloses a unit disk, and in particular that the unit circle encloses a minimal area among all such curves.

In this note we produce a family of bodies in $\mathbb R^3$ parameterized by $\varepsilon > 0$, each bounded by a smooth topological sphere with principal curvatures in $[-1, 1]$, and having volume arbitrarily close to
\[16 - 4\sqrt 3  + \left(10 \sqrt 3 - 14\right) \pi - \left(\frac{10}{3} - \sqrt 3\right) \pi^2 \approx 3.70. \]
 Thus, in contrast to $\mathbb R^2$, the unit sphere (which bounds a ball of volume $\frac{4 }{ 3} \pi \approx 4.19$) does not enclose the minimal  volume among all smooth spheres in $\mathbb R^3$ with principal curvatures in $[-1,1]$. This answers a folklore question of Dmitri Burago and Anton Petrunin.

\begin{figure}[htbp]
    \centering
    \includegraphics[height=5cm, keepaspectratio]{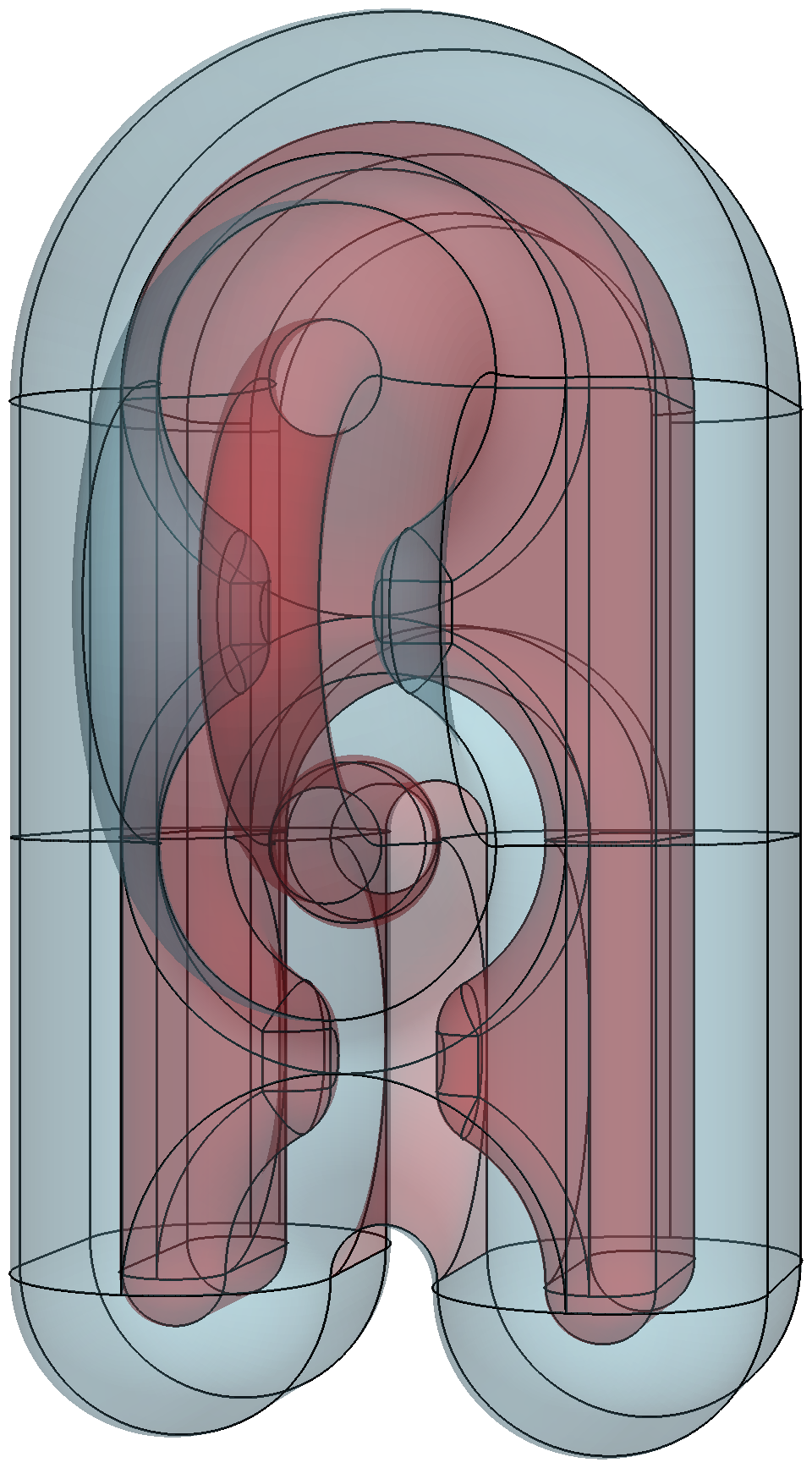}
    \hfill
    \includegraphics[height=5cm, keepaspectratio]{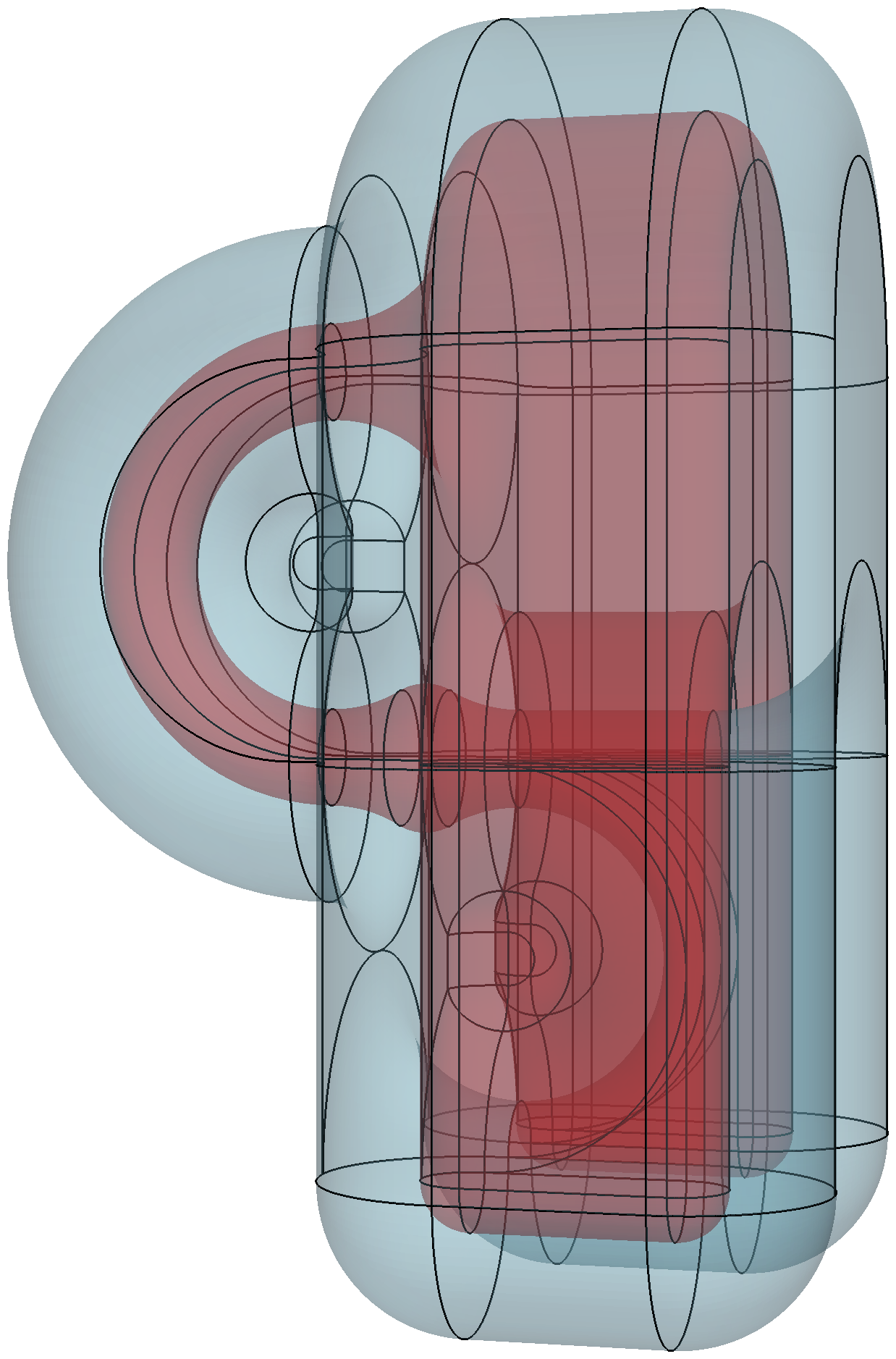}
    \hfill
    \includegraphics[height=5cm, keepaspectratio]{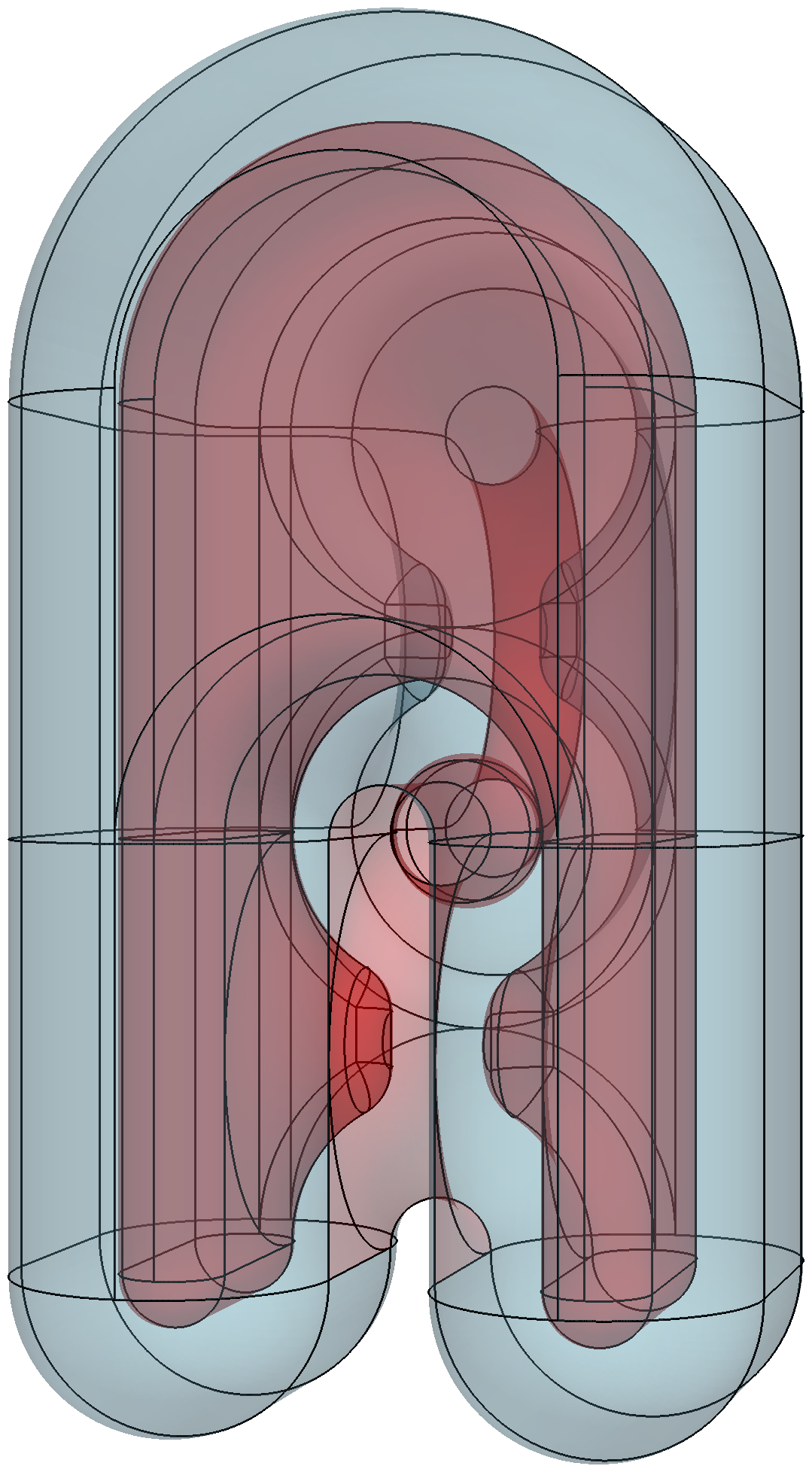}
    \caption{The member of the family with $\varepsilon = 0.5$, shaded red and blue.}
    \label{fig:genus_0_epsilon_half}
\end{figure}

\begin{figure}
    \centering
    \includegraphics[width=\linewidth]{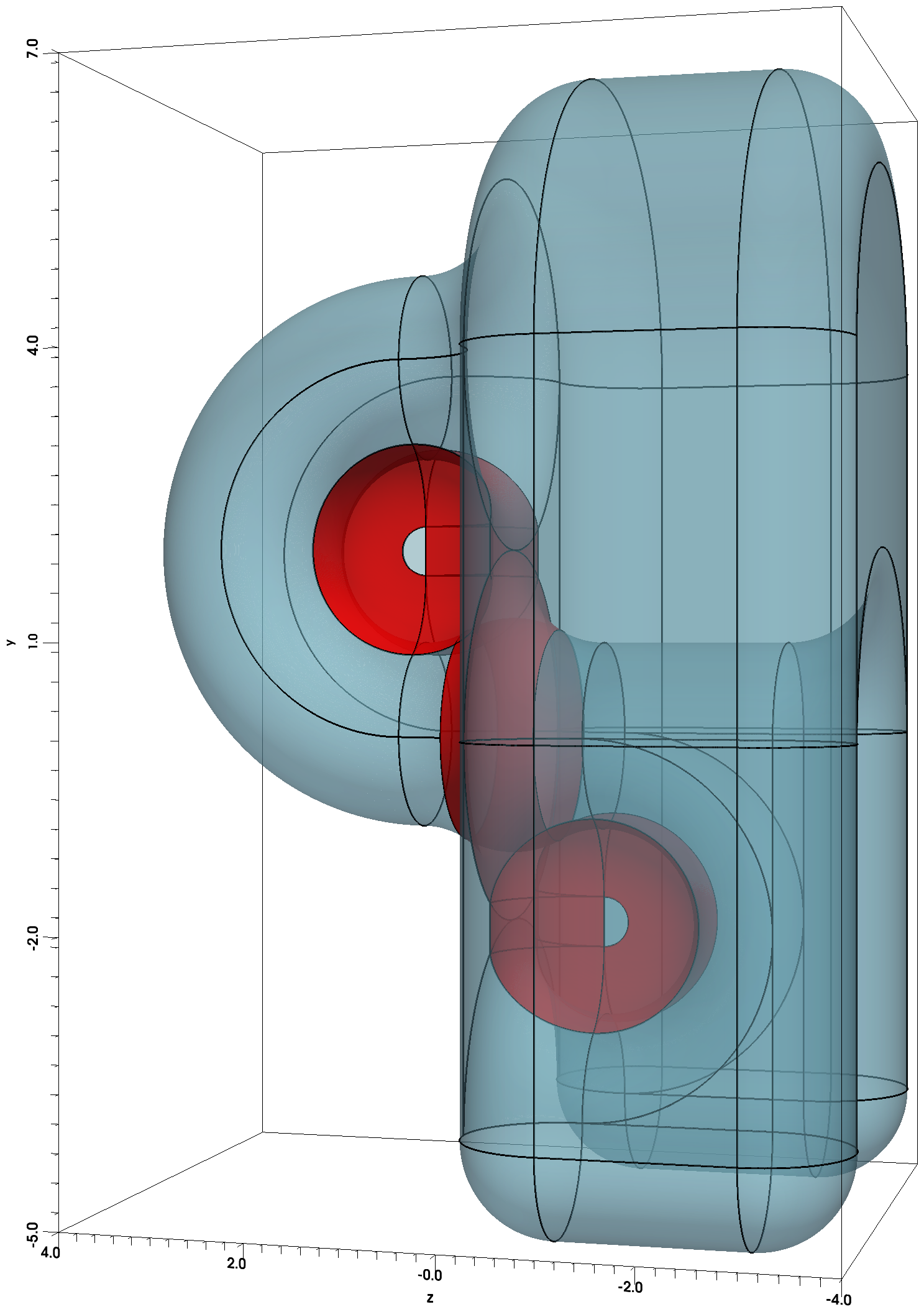}
    \caption{The limiting object ($\varepsilon = 0$) of the family, with a volume of approximately $3.70$ contained entirely in the red part. The blue part is a sheet of thickness $0$ where the inner cavity has run up against the outer walls, together with semicircular membranes of thickness $0$ in the center of the corks.}
    \label{fig:limiting_genus_0}
\end{figure}

An example with $\varepsilon = 0.5$ is shown in Figure~\ref{fig:genus_0_epsilon_half}, shaded red and blue. One can see the surface is a topological sphere by imagining pulling the red surfaces through the passage where they connect to the blue, though we later check this by computing the Euler characteristic. As one decreases $\varepsilon$, the red and blue parts of the surface get arbitrarily close together in most places, approaching the shape in Figure~\ref{fig:limiting_genus_0} which we will call the limiting body.

Members of our family will have a ``fat part" containing most of the volume, and a ``thin part" where two parts of the surface of the volume are within $2\varepsilon$ of each other (these are the red and blue parts of Figure~\ref{fig:limiting_genus_0} respectively). Rather than describing every surface in the ``thin part" twice and introducing an $\varepsilon$ to all equations, we will devote most of our effort to constructing the limiting body where $\varepsilon = 0$ and computing its volume.

In the Appendix there is a piecewise smooth parametric description of the surface resulting from thickening the limiting body by $\varepsilon$ and rescaling so that all principal curvatures are in $[-1, 1]$. This list has been symbolically verified to define a $C^1$ topological sphere by a SageMath script, which has been made available as an auxiliary file associated with the preprint on arXiv. Truly smooth examples having volume arbitrarily close to the limiting example can then be obtained by selecting a small enough $\varepsilon$ and applying standard smoothing processes. 

\subsection{History}
Lagunov \cite{lagunovinitialpapergenerallowerbound, lagunovgenerallowerbound} gave a generalization of the Pestov-Ionin theorem which when specialized to $\mathbb R^3$ shows that the inside of a closed surface $S$ contains a ball of radius $\frac{2}{\sqrt{3}} - 1 \approx 0.15$, and that this bound is tight \cite{lagunovgeneralupperbound} (see also \cite[p.115-117]{petrunin2024differentialgeometrycurvessurfaces}). Later Lagunov and Fet \cite{lagunovfetgenuslowerbound} (see \cite{lagunov2019extremalproblemssurfacesprescribed} for an English translation by Bishop) showed that if S is further required to be genus zero, then the inside of S must contain a larger ball of radius $\frac{\sqrt{6}}{2}-1\approx 0.22$, which is again tight \cite{lagunovfetgenusupperbound}.

Lagunov and Fet's results give a satisfactory extension of the Pestov-Ionin theorem to $\mathbb R^3$ in the sense that they accurately identify the largest radius of a sphere guaranteed to be contained within the enclosed volume of a surface with bounded curvature, both with and without the genus $0$ constraint. However, the examples they produce still enclose volume larger than the unit ball in $\mathbb R^3.$ Petrunin studied whether the volume enclosed by the surface S, not just the largest ball it encloses, could be less than a unit ball. This culminated in an example which has now become known as ``Petrunin's Fishbowl," a genus two surface enclosing a volume arbitrarily close to $\frac{22}{3} \pi - 2 \pi^2 \approx 3.30 < \frac{4}{3} \pi$ (see the appendix of \cite{qiu2025curvatureboundedsphereproblem}). The question of whether there could be a genus zero surface bounding less volume than the unit sphere was later popularized by Burago and Petrunin.

There has been a flurry of progress on this question in 2025. In \cite{qiu2025curvatureboundedsphereproblem}, Hongda Qiu shows that the result is true if the surface can be contained within an open ball of radius $2$, and that the surface in fact encloses a unit ball if the volume it encloses is star-shaped. He also suggests that a counterexample might be obtained by gluing on four disks to Petrunin's Fishbowl, in a sense adding two ``corks" to plug the holes. 

In the MathOverflow question \cite{petruninsMOQuestion}, Anton Petrunin asks whether such an example exists and outlines an improvement to his genus $2$ example due to Rostislav Matveev which attains volume arbitrarily close to $$2 \pi\left(2- \frac{\sqrt{3}}{3}\right)\left (\sqrt{3} - \frac{\pi }{2}\right) \approx 1.44.$$ Petrunin also includes calculations suggesting that a cork plugging a correctly shaped hole could theoretically have volume as small as $$
2\pi \left(2 - \frac{2}{\sqrt{3}}\right) \left(\sqrt{3} - \frac{\pi}{2} \right) \approx 0.85,$$
for a potential volume as low as $3.15.$ These estimates leave enough room to suggest that a counterexample should exist along the lines proposed in \cite{petruninsMOQuestion, qiu2025curvatureboundedsphereproblem}. However, neither the improved version of Petrunin's Fishbowl nor the proposed cork construction contains a ball of radius greater than $0.16 < 0.22$, so some work must be done to make these shapes compatible with Lagunov and Fet's results in the genus $0$ case \cite{lagunovfetgenuslowerbound}. 

Terence Tao, with some AI assistance, responded to Petrunin's question with another proof of impossibility in the cases of star-shaped bodies; a proof that there is some non-zero lower bound on the volume (which also follows from the work of Lagunov and Fet); and some useful discussion of the general shape of possible counterexamples \cite{taosMOAnswer}.

On October 17th, 2025, I posted a comment under Petrunin's question, containing a link to a 3-dimensional model\footnote{https://www.desmos.com/3d/p8xzkd6vv2 } of a genus $0$ counterexample in the web based graphing calculator Desmos, though I had to correct an incorrect surface on the 20th. On October 24th I posted an answer \cite{bolansMOAnswer} explaining this counterexample further and containing a corrected and slightly revised Desmos model\footnote{https://www.desmos.com/3d/hwgiep1wfa}. 

The counterexample, which will be the subject of this paper, follows the suggestion to modify Matveev's improved genus 2 example by introducing two corks. The main contributions are a cork design that connects nicely to the genus 2 example, and a modification of the genus 2 example to accommodate not one but two locations where these corks fit. 

Also on October 24th, Hongda Qiu posted a counterexample of his own to the MathOverflow thread \cite{qiusMOAnswer}. After he implemented some corrections, the resulting construction appears to have used an identical ``volume-containing" part to the construction we present here. On November 16th, Qiu posted a preprint \cite{qiu2025supersqueezedsphere} going further into the history of the problem and giving a qualitative description of his construction.

\section{A modified Petrunin's Fishbowl}

After Matveev's improvement, the limiting object of Petrunin's Fishbowl family, shown in Figure~\ref{fig:petrunin_limiting}, contains a ``handle" of the shape depicted in Figure~\ref{fig:top_handle_unfilled} (for parameterizations defining the handle see Section~\ref{section:cork}). This handle provides a natural place to try to plug with a cork. In Section \ref{section:cork} we will construct a cork of volume approximately $1.13$ that will fit nicely into this handle --- see Figure~\ref{fig:plugged_handle} for a picture. If we only plug this handle then we obtain an easy genus $1$ example beating the unit ball. However, to obtain a genus $0$ example we must find a second efficient location to place a cork. There is no easy location that will both keep the volume below that of the unit ball and result in a figure with a connected surface, so we instead must modify the example to accommodate one.

\begin{figure}[H]
    \centering
    \includegraphics[width=0.5\linewidth]{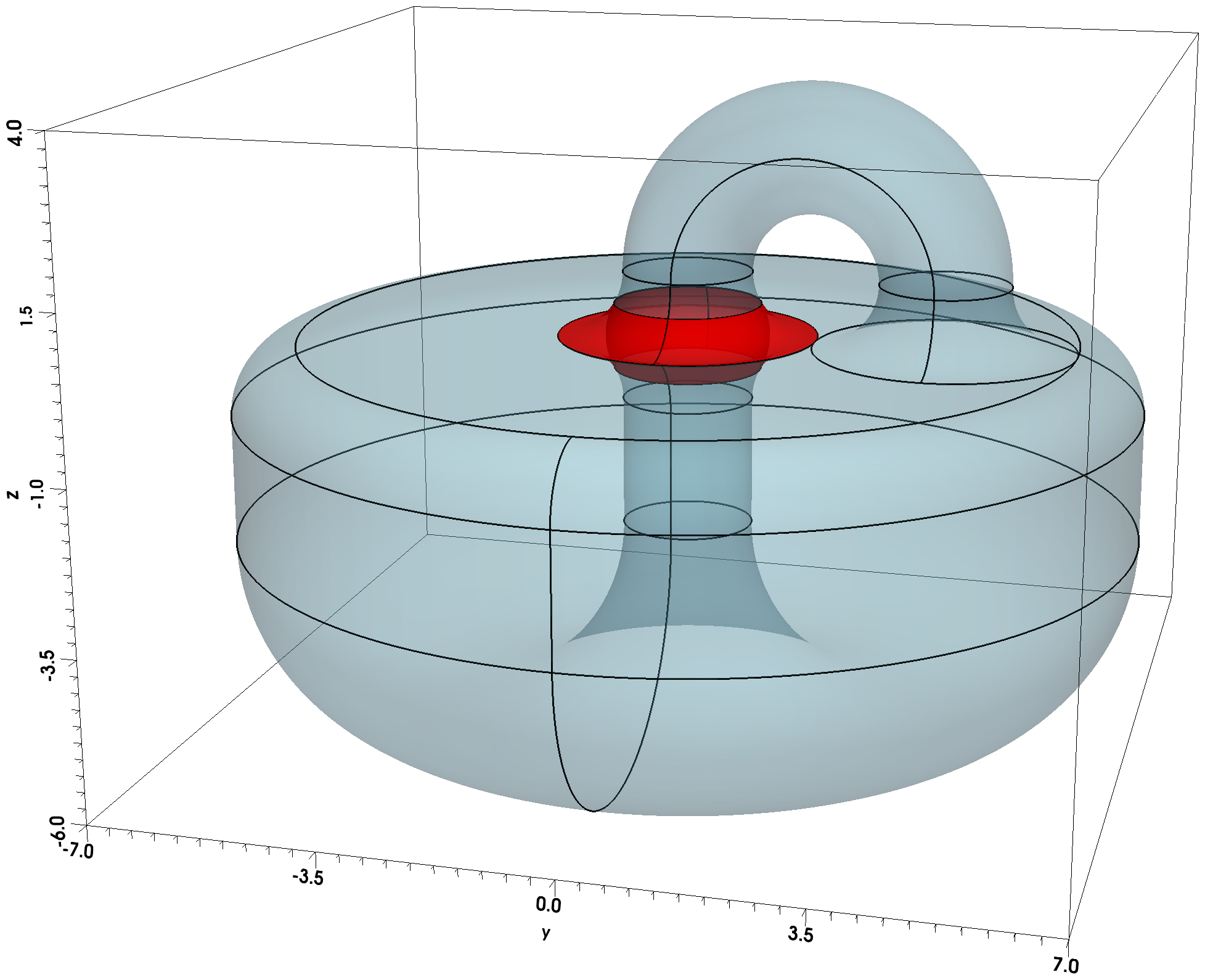}
    \caption{The limiting object of Petrunin's Fishbowl after Matveev's improvement, following the description and models in \cite{petruninsMOQuestion}. It contains a volume of approximately 1.44 entirely in the red part, while the blue part is thickness $0$ in the limit. As always, true examples are obtained by thickening and smoothing.}
    \label{fig:petrunin_limiting}
\end{figure}

\begin{figure}[H]
    \centering
    \includegraphics[width=0.5\linewidth]{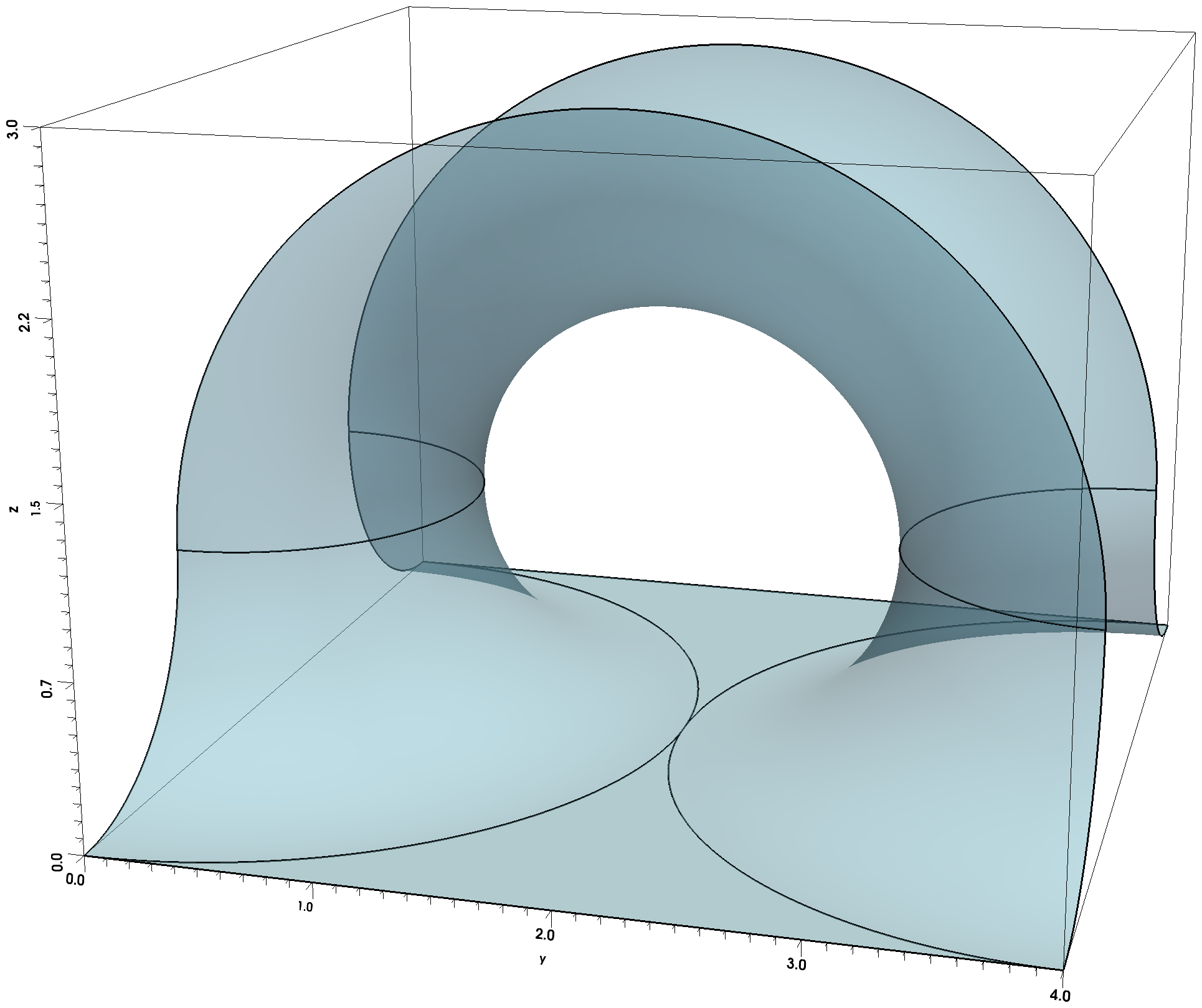}
    \caption{The top handle of Petrunin's Fishbowl. It is comprised of three parts of tori, each of major radius $2$ and minor radius $1$, and part of a plane.}
    \label{fig:top_handle_unfilled}
\end{figure}
A key insight is that it is possible to squeeze the bottom of the fishbowl to accommodate a second handle --- see Figure~\ref{fig:squeezing} for how we do this. So that we do not need to invent two distinct corks, we have arranged things to match the shape of the original handle where the cork will attach. In fact, the new handle is a reflection of the old handle about the origin.

\begin{figure}[H]
    \centering
    \includegraphics[width=0.4\linewidth]
    {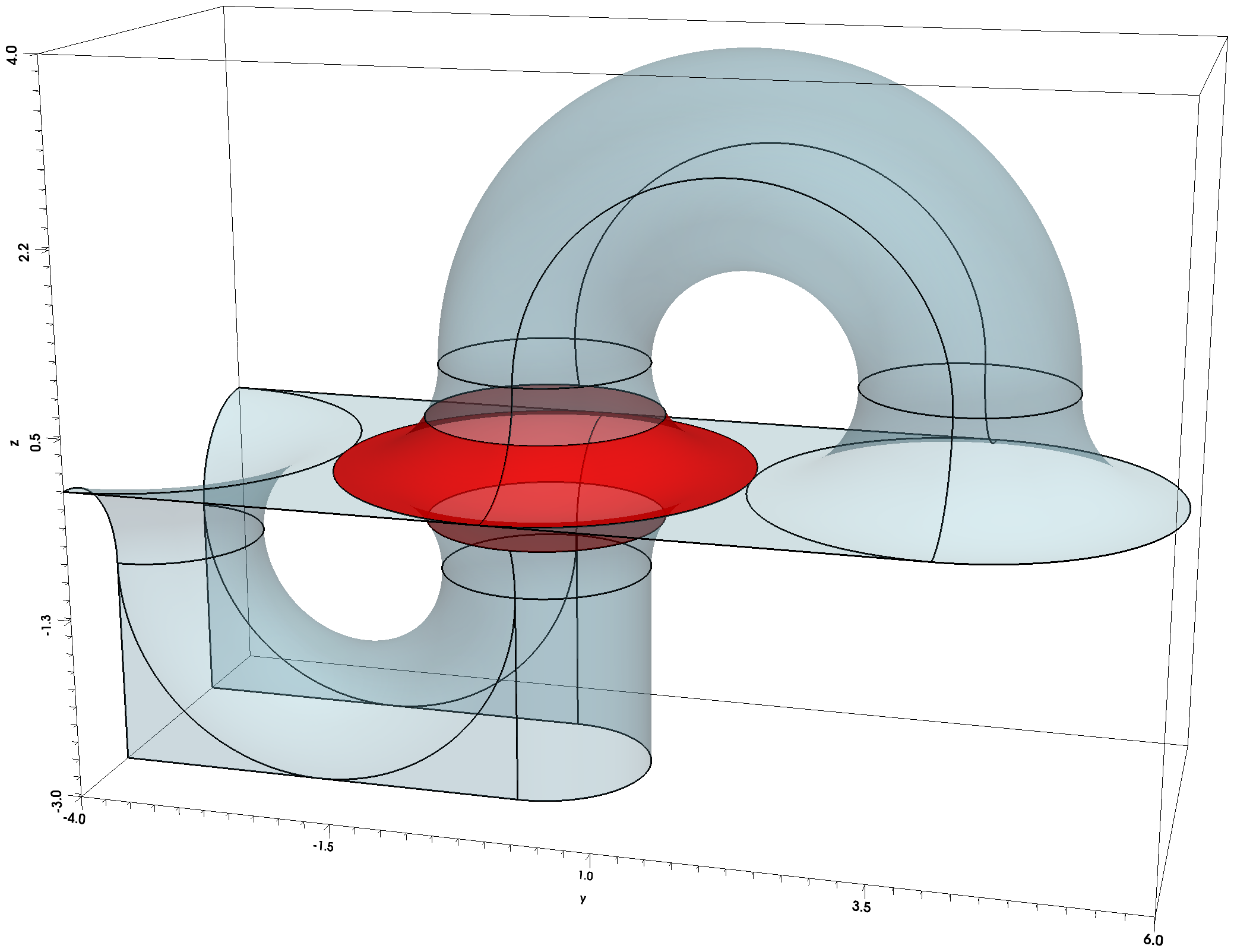}
    \includegraphics[width=0.4\linewidth]{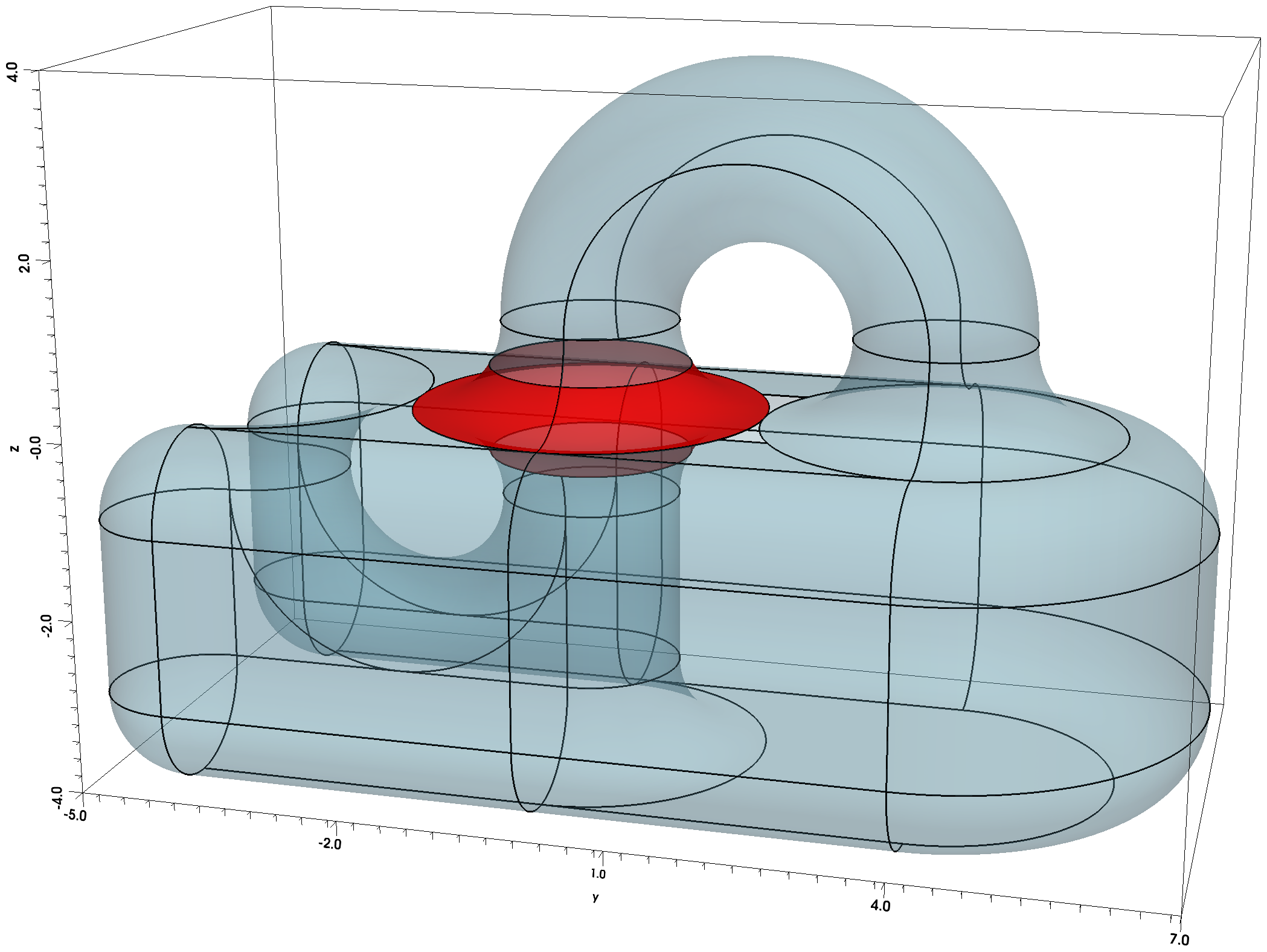}
    \caption{The result of squeezing Petrunin's Fishbowl to create a lower handle matching the upper, both with and without the part of the surface closing up the shape. The only part with non-zero volume in the limit is shown in red.}
    \label{fig:squeezing}
\end{figure}

For actual parameterizations of the surfaces describing this modification, see the Appendix. Though there are many surfaces involved, every blue surface used in Figure~\ref{fig:squeezing} is a part of a plane, sphere, cylinder, or else a torus of minor radius $1$ and major radius at least $2$. These meet the curvature bounds automatically, and their simple geometry makes arranging the $C^1$ condition easy. Confirmation that the resulting surface is indeed $C^1$ was done via the SageMath script in the auxiliary files of the arXiv submission. The closing surfaces have been chosen to be relatively small, made from simple surfaces, and to result in a final shape similar to Petrunin's original example to make the squeezing process clear.

\begin{figure}[H]
    \centering
    \input{deltoid_washer}
    \includegraphics[width=0.5\linewidth]{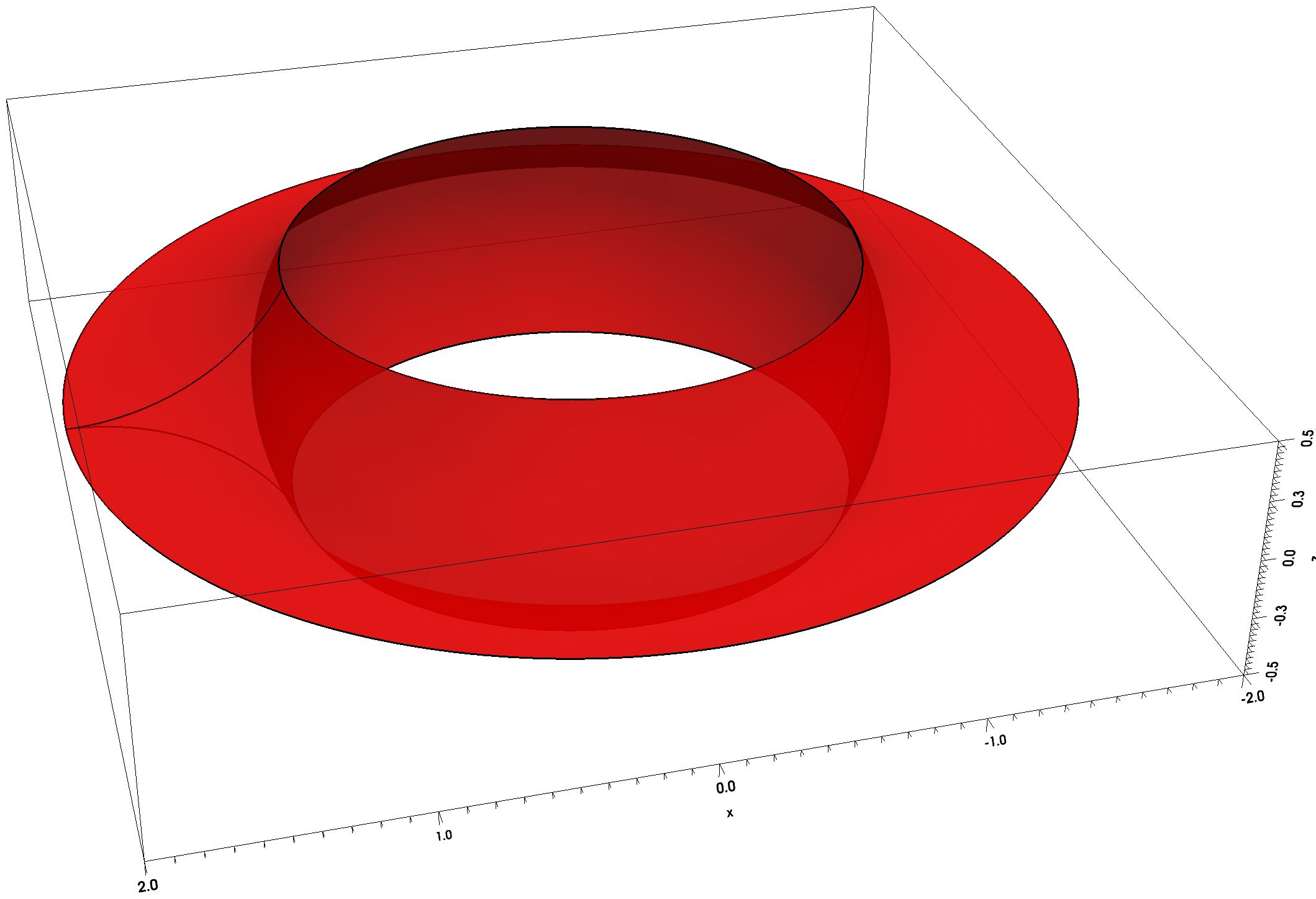}
    \caption{The only part of Petrunin's Fishbowl with any volume, a solid of revolution of volume approximately $1.44.$ Its shape is obtained by revolving the deltoid between the three mutually tangent unit circles centered at $(2 - \sqrt{3}, 0), (2, 1)$ and $(2, -1)$ around the $y$-axis.}
    \label{fig:genus_2_volume}
\end{figure}
The shape and volume of the fat part --- see Figure~\ref{fig:genus_2_volume} --- has remained unchanged compared to the original fishbowl and so has already been computed by Petrunin in \cite{petruninsMOQuestion}. It is congruent to the solid of revolution about the $y$ axis depicted in Figure \ref{fig:genus_2_volume}. This solid of revolution is swept out by a deltoid of area $(\sqrt{3} - \frac{\pi}{2})$ with centroid $(2 - \frac{\sqrt{3}}{3}, 0)$, and so has volume \begin{equation}
    2 \pi\left(2- \frac{\sqrt{3}}{3}\right)\left (\sqrt{3} - \frac{\pi }{2}\right) \approx 1.44 \label{eq:Petrunin_Volume}
\end{equation}
by Pappus's centroid theorem.
The arcs $\overset{\frown}{P_1 P_2}$ and $\overset{\frown}{P_1 P_3}$ sweep out parts of tori of major radius $2$ and minor radius $1$, and so these surfaces meet the curvature bound. The arc $\overset{\frown}{P_2 P_3}$ sweeps out a surface corresponding to  \begin{equation}
    \left(-\left(r_{1}+\sin\left(v\right)\right)\cos\left(u\right),-\left(r_{1}+\sin\left(v\right)\right)\sin\left(u\right),-\cos\left(v\right)\right)
\end{equation} in the actual fishbowl,
where $r_1 = 2 - \sqrt{3}$, $\frac{\pi}{3} \le v \le \frac{2 \pi}{3}$ and $0 \le u \le 2 \pi$. This has principal curvatures $1$ and $\frac{\sin(v)}{r_1 + \sin(v)}$, which obey the needed bounds as $0 \le \sin(v) \le r_1 + \sin(v)$ in the given range (Alternately, one could observe that this surface is part of the positively curved part of a spindle torus with minor radius $1$, which always has principal curvatures in $[0,1]$ when oriented correctly).

Armed with this example, it remains to demonstrate an appropriate cork for this handle.

\section{Construction of a cork}

\label{section:cork}
It remains to plug the two ``handles" of our modified fishbowl. As we have arranged the handles to be of the same shape and symmetric about the origin, it suffices to construct a single solid cork $C$ to insert into the first handle. By symmetry the second handle can then be plugged with $-C$. To define $C$ we construct a surface $S_1$ (see Figure~\ref{fig:cork_surface} for a picture) which can be glued cleanly onto the handle, and let $S_2$ be its reflection through the $x=0$ plane. The cork $C$ will be the body bounded by the surfaces $S_1$, $S_2$ and the handle, together with a thickness $0$ semicircular membrane where $S_1$ and $S_2$ coincide. 

Concretely, the first handle we plug consists of
\begin{itemize}
    \item The toroidal patch with $0 \le u \le \pi$ and $\pi \le v  \le 2 \pi$ defined by
\begin{equation}
\left(\cos\left(v\right),\left(2+\sin\left(v\right)\right)\cos\left(u\right)+2,\left(2+\sin\left(v\right)\right)\sin\left(u\right)+1\right) \label{eq:backpack}
\end{equation}
\item The toroidal patch with $0 \le u \le \pi$ and $\pi \le v \le \frac{3 \pi}{ 2}$ defined by
\begin{equation} \left(\left(2+\sin\left(v\right)\right)\cos\left(u\right),\left(2+\sin\left(v\right)\right)\sin\left(u\right),\cos\left(v\right)+1\right) \label{eq:backpack_bottom_base}
\end{equation}

\item the toroidal patch with $\pi \le u \le 2 \pi$ and $\pi \le v \le \frac{3 \pi}{ 2}$ defined by
\begin{equation}
\left(\left(2+\sin\left(v\right)\right)\cos\left(u\right),\left(2+\sin\left(v\right)\right)\sin\left(u\right)+4,\cos\left(v\right)+1\right) \label{eq:backpack_top_base}
\end{equation}
\item the part of the plane
\begin{equation}z=0. \label{eq:back}\end{equation}
obeying $x^2 + (y-4)^2 \ge 4$; $x^2 + y^2 \ge 4,$ $-2 \le x \le 2$ and $0 \le y \le 4$.
\end{itemize}
 See Figure~\ref{fig:top_handle_unfilled} for a picture of this handle.

For convenience set $r_1 = 2 - \sqrt 3$. We now define a surface $S_1$ which can be glued cleanly into this handle. It consists of five parts:
\begin{itemize}

    \item The toroidal patch  with $0 \le u \le \pi$ and $\frac{2\pi}{3} \le v \le \pi$ defined by \begin{equation} \label{eq:cork_surface_torus}
\left(\cos\left(v\right)+1,\left(r_{1}+\sin\left(v\right)\right)\cos\left(u\right)+2,\left(r_{1}+\sin\left(v\right)\right)\sin\left(u\right)+1\right)
    \end{equation} 
    \item The semicircle  with $0 \le u \le \pi$ and $0 \le v \le r_1$ defined by \begin{equation} \label{eq:semicircle}
\left(0,v\cos\left(u\right)+2,v\sin\left(u\right)+1\right)
    \end{equation}
    \item The $\frac{1}{6}$th of a unit hemisphere with $0 \le u \le \frac{\pi}{2}$ and $\frac{2 \pi}{3} \le v \le \pi$ defined by \begin{equation}
        \left(\cos\left(v\right)\cos\left(u\right)+1,-\sin\left(v\right)\cos\left(u\right)+\sqrt{3},-\sin\left(u\right)+1\right)
    \end{equation}.
    \item The $\frac{1}{6}$th of a unit hemisphere  with $0 \le u \le \frac{\pi}{2}$ and $\pi \le v \le \frac{4 \pi}{3}$ defined by \begin{equation}
        \left(\cos\left(v\right)\cos\left(u\right)+1,-\sin\left(v\right)\cos\left(u\right)+4-\sqrt{3},-\sin\left(u\right)+1\right)
    \end{equation}
    \item The cylindrical patch with $-r_1 \le u \le r_1$ and $\frac{\pi}{2} \le v \le \pi$ defined by \begin{equation}
        \left(\cos\left(v\right)+1,-u+2,-\sin\left(v\right)+1\right).
    \end{equation}
\end{itemize}
See Figure~\ref{fig:cork_surface} for a picture of this surface, and Figure~\ref{fig:plugged_handle} for how it fits into the handle.

\begin{figure}
    \centering
    \includegraphics[width=0.5\linewidth]{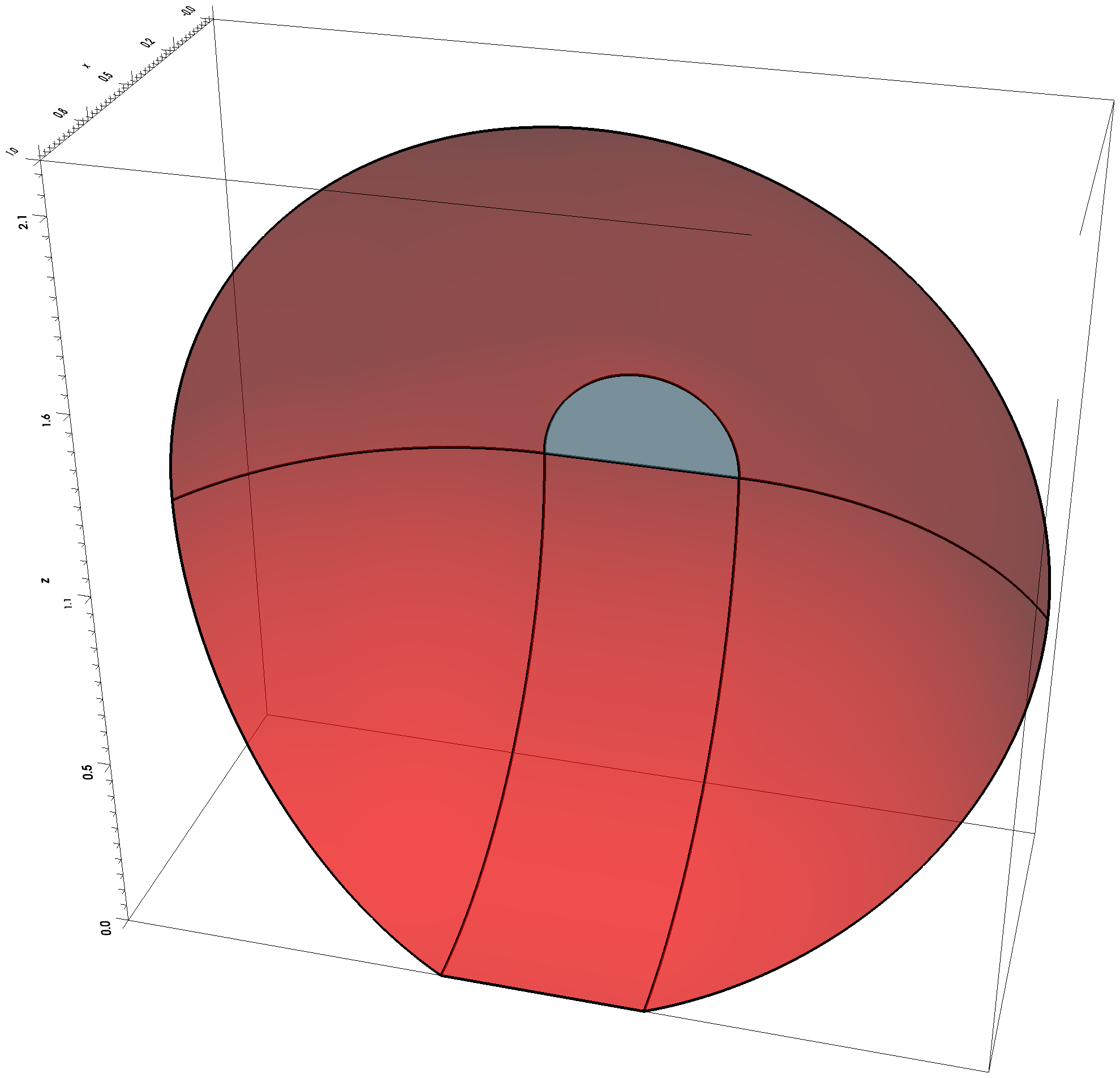}
    \caption{The 5-part topological disk $S_1$ that makes up one side of the cork.}
    \label{fig:cork_surface}
\end{figure}

The first of these surfaces \eqref{eq:cork_surface_torus} has principal curvatures at $(u,v)$ of $\frac{\sin v}{r_1 + \sin v}$ and $1$. These are in $[0,1]$ since $0 \le u \le \pi$. All the other surfaces are parts of flat planes, unit spheres, and unit cylinders, and so obey the curvature condition automatically. One further checks that the underlying surfaces are tangent to each other and the surface of the handle along the appropriate boundaries. As always, these checks were done with the SageMath script in the auxiliary files of the arXiv submission.

\begin{remark}
$S_1$ is part of the surface consisting of points of distance $1$ from the solid semicircle $U' = U + (1,0,0)$, where $U$ is the semicircle \eqref{eq:semicircle}, from which one can confirm the curvature and $C^1$ conditions manually. This structure is no accident: unit spheres with center on $U'$ are all tangent to $U$, and those centered on $\partial U'$ are additionally tangent to the handle. 
\end{remark}

We let $S_2$ be the reflection of $S_1$ across the plane $x=0$. Note that the semicircle \eqref{eq:semicircle} gets reflected into itself. This is intentional as this semicircle is part of the ``thin" part of the body, and when the shape is finally thickened by $\varepsilon$ this will become two semicircles separated by $2 \varepsilon$. 

We let $C$ be the union of $S_1, S_2$, and the solid body bound between them and the handle. Monte Carlo integration estimates the volume of $C$ at around $1.129$. We now confirm this estimate by computing the volume exactly.

\begin{figure}[htbp]
    \centering
    \begin{subfigure}[b]{0.45\textwidth}
        \label{fig:cork_volume}
        \centering
        \includegraphics[height=2.8cm, keepaspectratio]{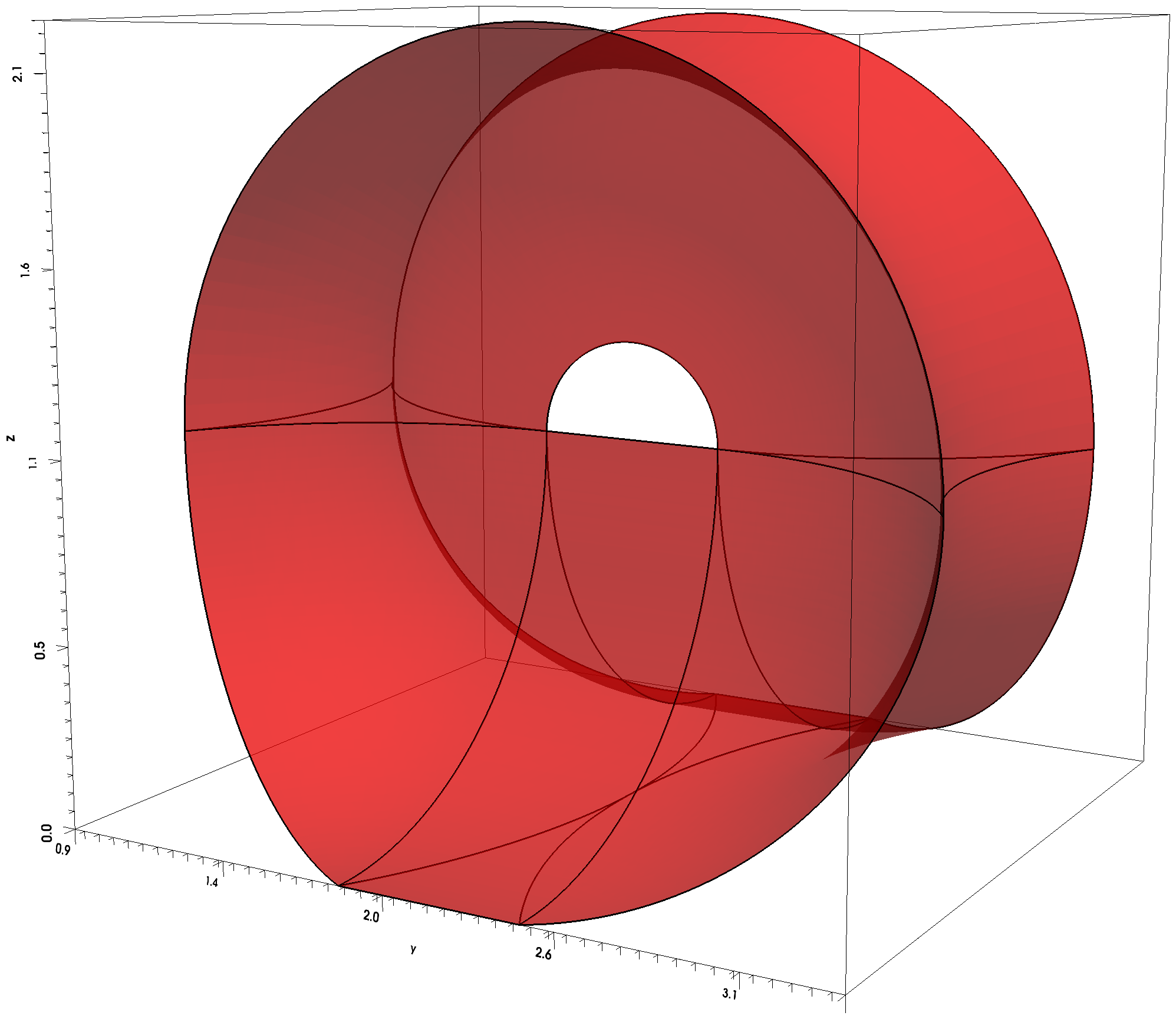}
        \caption{The part of the cork with non-zero volume.}
    \end{subfigure}
    \hfill
    \begin{subfigure}[b]{0.45\textwidth}
        
        \centering
        \includegraphics[height=2.8cm, keepaspectratio]{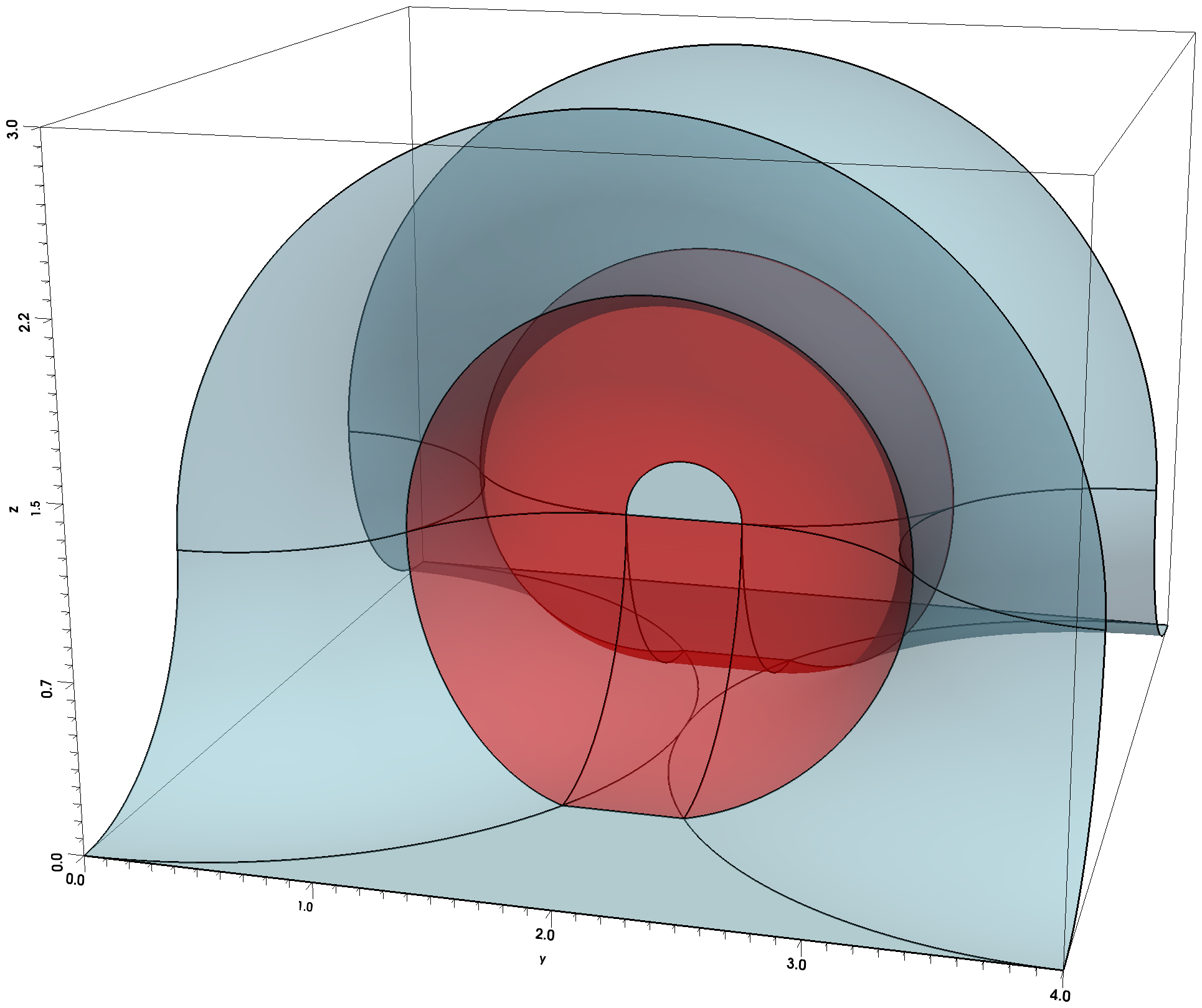}
        \caption{The handle after plugging it with the cork}
        \label{fig:plugged_handle}
    \end{subfigure}
    \caption{The cork and how it fits into the handle.}
    \label{fig:cork_and_handle}
\end{figure}

\begin{lemma}
\label{lem:cork_volume}
The solid ``cork" $C$ has volume $$
\left(8 - 2\sqrt 3\right) - \left(6 - 3\sqrt 3\right) \pi - \left(\frac{2}{3} - \frac{1}{\sqrt{3}}\right) \pi^2 \approx 1.129
$$
\end{lemma}

\begin{figure}
    \centering
    \input{deltoid_cork}
    \caption{The top half of the cork is congruent to half of the solid of revolution obtained by revolving the deltoid between the three mutually tangent unit circles centered at $(r_1, 1), (r_1, -1)$ and $(2, 0)$ about the $y$-axis.}
    \label{fig:cork_revolution}
\end{figure}
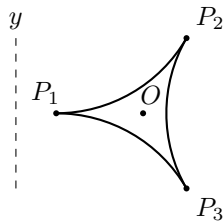

\begin{proof}
We divide $C$ into two parts, one with $1 \le z \le 2$, and the other with $0 \le z \le 1$. The first part is bound by $z=1$ and the three toroidal patches \eqref{eq:backpack}, \eqref{eq:cork_surface_torus}, and the reflection of $\eqref{eq:cork_surface_torus}$ across $x=0.$ As the toroidal patches have a common major axis, this is just half of the solid of revolution given by revolving the figure \ref{fig:cork_revolution} around the marked axis (Petrunin referred to this solid of revolution as the ``optimal cork" \cite{petruninsMOQuestion}). By Pappus's centroid theorem, the volume of this part is then $$
\pi \left(2 - \frac{2}{\sqrt{3}}\right) \left(\sqrt{3} - \frac{\pi}{2} \right) \approx 0.428.$$

For the bottom half, we focus on the region contained in the polytope cut out by $0 \le x \le 1$, $0 \le y \le 2,$ $0 \le z \le 1,$ and $\sqrt{3} x < y,$ as the whole bottom half consists of $4$ such regions (the cork is symmetric about both of the planes $y=2$ and  $x=0$). The volume of the cork within this polytope is equal to the volume of the polytope, which is $2 - \frac{\sqrt{3}}{2}$, minus the following four volumes:
\begin{itemize}
    \item The pie slice of points above the cylindrical surface, cut out by $$
\left(x-1\right)^{2}+\left(z-1\right)^{2}\le1;\:0\le x\le1;\:\sqrt{3}\le y\le2;\:0\le z\le1
$$
of volume $\frac{\pi}{4} (2 - \sqrt{3})$
\item The $\frac{1}{12}$th of a sphere above the spherical surface, cut out by $$
\left(x-1\right)^{2}+\left(z-1\right)^{2}+\left(y-\sqrt{3}\right)^{2}\le1;\:z\le1;\:y\le\sqrt{3};\:\sqrt{3}x\le y
$$
of volume $\frac{\pi}{9}$
\item The pie slice of points close to the origin cut out by $$
x^{2}+y^{2}\le1; \:\sqrt{3}x\le y;\:0\le x;\:1\ge z\ge0
$$
of volume $\frac{\pi}{12}$
\item The $\frac{1}{12}$th of a solid of revolution below the handle, cut out by \[
\left(\sqrt{x^{2}+y^{2}}-2\right)^{2}+\left(z-1\right)^{2}\ge1^{2};\]\[\:\sqrt{3}x\le y;\:0\le x;\:1\le x^{2}+y^{2}\le4;\:1\ge z\ge0
\]
which has volume $\frac{\pi}{6}(1 + \frac{10 - 3 \pi}{12 - 3\pi})(1-\frac{\pi}{4}) = \frac{1}{36} (11 - 3 \pi) \pi$.
\end{itemize}
Thus the bottom half of the cork has volume $$
4\left(\left(2 - \frac{\sqrt{3}}{2}\right) - \frac{\pi}{4} (2 - \sqrt{3}) - \frac{\pi}{9} - \frac{\pi}{12} - \frac{1}{36} (11 - 3 \pi) \pi \right) \approx 0.701.
$$
Combining the volumes for the top and bottom of the cork yields the claimed expression for the volume.
\end{proof}

\begin{remark}
    The shape of the cork and handle combination has not been optimized for volume, for example one can depress parts of the plane $z=0$ at the base of the handle into the cork to shave off some volume.
\end{remark}

\section{The final assembly}

The limiting member of our family is now obtained by plugging both handles of the modified genus $2$ example with the corks, resulting in the limiting body of Figure~\ref{fig:limiting_genus_0} (compare with the genus $2$ body of Figure~\ref{fig:squeezing}).
\begin{theorem}
There exist bodies in $\mathbb R^3$ bounded by a smooth topological sphere with principal curvatures in $[-1,1]$, and volumes arbitrarily close to \[16 - 4\sqrt 3  + \left(10 \sqrt 3 - 14\right) \pi - \left(\frac{10}{3} - \sqrt 3\right) \pi^2 \approx 3.70\]
\end{theorem}

\begin{proof}
    The part of the limiting example of Figure~\ref{fig:limiting_genus_0} containing non-zero volume --- isolated in Figure~\ref{fig:genus_0_volume} --- consists of the corks $C$ and $-C$, whose volume we computed in Lemma~\ref{lem:cork_volume}, and the washer with volume given by \eqref{eq:Petrunin_Volume}. Combining yields the claimed expression.

    The curvature bound follows for the limiting example by construction: All of our surfaces are either a toroidal patch arising from the solids of revolution discussed in the previous two sections, or are parts of planes; cylinders of radius at least one; spheres of radius one; or tori of major radius at least two and minor radius one. The SageMath script in the auxiliary files of the arXiv submission confirms this symbolically, along with the fact that the surfaces are glued together in a $C^1$ fashion.

    To obtain a true example, we must fatten the thin region to be truly $3$-dimensional and make the surface of the result smooth. To do this, one considers the set of all points of distance $0.5 > \varepsilon > 0$ to the limiting example, resulting in a surface with principal curvatures in $[\tfrac{-1}{1- \varepsilon}, \tfrac{1}{1- \varepsilon}] \subseteq [1 - 2 \varepsilon, 1 + 2 \varepsilon]$. A standard mollification argument (using that being $C^1$ with bounded principal curvatures implies being $C^{1,1}$) allows us to approximate this thickening by a smooth surface while keeping the principal curvatures in $[-1-3 \varepsilon, 1 + 3\varepsilon]$. Scaling by $1+ 3\varepsilon$ so that the principal curvatures of the result lie in $[-1,1]$, we obtain a smooth example.  By choosing $\varepsilon$ arbitrarily small, we thus obtain smooth examples having volume arbitrarily close to the volume of the limiting example of Figure~\ref{fig:limiting_genus_0}. 
    
    The SageMath script also confirms that the surface of the body consists of a single connected component of Euler characteristic $2$ and without boundary, so it is indeed a sphere. One can also see this from the depiction in Figure~\ref{fig:genus_0_epsilon_half} by imagining pulling the red surfaces out through the passage where they join to the blue ones.
\end{proof}
\begin{figure}
    \centering
    \includegraphics[width=0.5\linewidth]{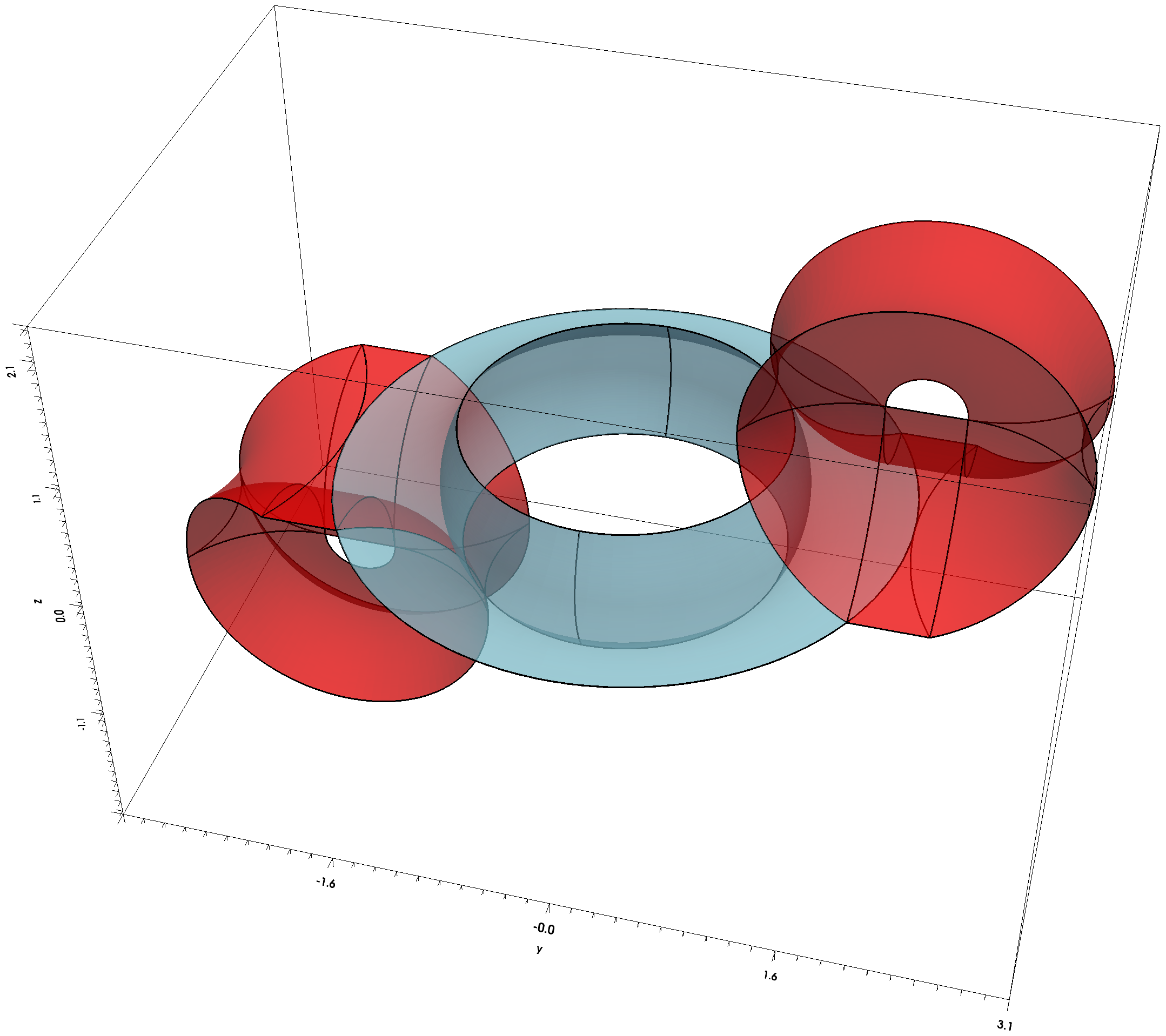}
    \caption{The only parts of the limiting object having non-zero volume, the two red corks and blue washer. Note that the corks and washer are connected to each other, as a part of the handles we glued the corks into borders the washer.}
    \label{fig:genus_0_volume}
\end{figure}
\appendix
\section{Appendix}

We list all surfaces involved in the final construction after thickening by $\varepsilon$ and rescaling by $\frac{1}{1 - \varepsilon}$ so that the principal curvatures lie in $[-1, 1]$ (but we have not done any mollification). Many surfaces have been subdivided into smaller pieces than is necessary for a human understanding of the example --- this is to make the computer verification easier by avoiding edges where one surface glues to several others. The use of the equations describing the thickened body and not just the limiting one is to remove some ambiguity in how the surfaces should glue together, as in the limit some surfaces become identical. Edges of a surface in the list always correspond to a single extreme values of one of the parameters, but some surfaces have fewer than four edges as they are constant at one or more of the extreme values. In total there are $110$ surfaces, $208$ edges, and $100$ vertices.

This list is also present within the SageMath script contained in the auxiliary materials associated to the arXiv submission.

\input{table}

\printbibliography

\end{document}

%% file: deltoid_washer.tex
\begin{tikzpicture}[scale=2]

  \draw[dashed] (0,-0.5) -- (0,0.5) node[above] {$y$};

  \coordinate (C1) at ({2-sqrt(3)},0);
  \coordinate (C2) at (2,1);
  \coordinate (C3) at (2,-1);

  \coordinate (P1) at (2,0);
  \coordinate (P2) at ({2 - sqrt(3)/2}, 0.5);
  \coordinate (P3) at ({2 - sqrt(3)/2}, -0.5);

  \coordinate (O) at ({2 - sqrt(3)/3},0);

  \draw[thick] (C2) ++(210:1) arc[start angle=210, end angle=270, radius=1];
  \draw[thick] (C3) ++(90:1) arc[start angle=90, end angle=150, radius=1];
  \draw[thick] (C1) ++(-30:1) arc[start angle=-30, end angle=30, radius=1];

  \fill (P1) circle (0.02) node[above right] {$P_1$};
  \fill (P2) circle (0.02) node[above left] {$P_2$};
  \fill (P3) circle (0.02) node[below left] {$P_3$};

  \fill (O) circle (0.02) node[above left, xshift=5, yshift=-0.5] {$O$};
\end{tikzpicture}

%% file: deltoid_cork.tex
\begin{tikzpicture}[scale=2]

  \draw[dashed] (0,-0.5) -- (0,0.5) node[above] {$y$};

  \coordinate (A) at (2,0);
  \coordinate (B) at ({2-sqrt(3)},1);
  \coordinate (C) at ({2-sqrt(3)},-1);
  \coordinate (O) at ({2-sqrt(3)/3},-1);

  \coordinate (P1) at ({2-sqrt(3)},0);
  \coordinate (P2) at ({1-sqrt(3)/2 + 1},0.5);
  \coordinate (P3) at ({1-sqrt(3)/2 + 1},-0.5);

  \coordinate (O) at ({2 - 2*sqrt(3)/3},0);

  \draw[thick] (A) ++(150:1) arc[start angle=150,end angle=210,radius=1];
  \draw[thick] (B) ++(-30:1) arc[start angle=-30,end angle=-90,radius=1];
  \draw[thick] (C) ++(90:1) arc[start angle=90,end angle=30,radius=1];

  \fill (P1) circle (0.02) node[above left, xshift=5pt] {$P_1$};
  \fill (P2) circle (0.02) node[above right] {$P_2$};
  \fill (P3) circle (0.02) node[below right] {$P_3$};

  \fill (O) circle (0.02) node[above, xshift=3pt] {$O$};
\end{tikzpicture}

%% file: table.tex
{\tiny
\setlength{\tabcolsep}{1pt}
\begin{longtable}{c c c c c c}
\caption{{\normalsize List of all surfaces in the parametric form $f(u,v)$ for $u_0 \le u \le u_1$ and $v_0 \le v \le v_1.$} We also set $r = 2 - \sqrt{3}$.} \\
\toprule
{\normalsize ID} &{\normalsize $f(u,v)$} & {\normalsize $u_{0}$} & {\normalsize $u_{1}$} & {\normalsize $v_{0}$} & {\normalsize $v_{1}$} \\
\midrule
\endfirsthead
\toprule
Surface & Equation & $u_{\min}$ & $u_{\max}$ & $v_{\min}$ & $v_{\max}$ \\
\midrule
\endhead
0 & $ \left(\frac{{\left(\mathit{\varepsilon} - 1\right)} \cos\left(u\right) \cos\left(v\right) - 1}{\mathit{\varepsilon} - 1}, \frac{{\left(\mathit{\varepsilon} - 1\right)} \cos\left(u\right) \sin\left(v\right) + \sqrt{3}}{\mathit{\varepsilon} - 1}, \frac{{\left(\mathit{\varepsilon} - 1\right)} \sin\left(u\right) + 1}{\mathit{\varepsilon} - 1}\right) $ & $0$ & $\frac{1}{2} \pi$ & $\frac{2}{3} \pi$ & $\pi$ \\
1 & $ \left(-\frac{{\left(\mathit{\varepsilon} - 1\right)} \cos\left(u\right) \cos\left(v\right) - 1}{\mathit{\varepsilon} - 1}, \frac{{\left(\mathit{\varepsilon} - 1\right)} \cos\left(u\right) \sin\left(v\right) + \sqrt{3}}{\mathit{\varepsilon} - 1}, \frac{{\left(\mathit{\varepsilon} - 1\right)} \sin\left(u\right) + 1}{\mathit{\varepsilon} - 1}\right) $ & $0$ & $\frac{1}{2} \pi$ & $\frac{2}{3} \pi$ & $\pi$ \\
2 & $ \left(\frac{{\left(\mathit{\varepsilon} - 1\right)} \cos\left(v\right) - 1}{\mathit{\varepsilon} - 1}, -\frac{u - 2}{\mathit{\varepsilon} - 1}, \frac{{\left(\mathit{\varepsilon} - 1\right)} \sin\left(v\right) + 1}{\mathit{\varepsilon} - 1}\right) $ & $-r$ & $r$ & $\frac{1}{2} \pi$ & $\pi$ \\
3 & $ \left(-\frac{{\left(\mathit{\varepsilon} - 1\right)} \cos\left(v\right) - 1}{\mathit{\varepsilon} - 1}, -\frac{u - 2}{\mathit{\varepsilon} - 1}, \frac{{\left(\mathit{\varepsilon} - 1\right)} \sin\left(v\right) + 1}{\mathit{\varepsilon} - 1}\right) $ & $-r$ & $r$ & $\frac{1}{2} \pi$ & $\pi$ \\
4 & $ \left(\frac{{\left(\mathit{\varepsilon} - 1\right)} \cos\left(u\right) \cos\left(v\right) - 1}{\mathit{\varepsilon} - 1}, \frac{{\left(\mathit{\varepsilon} - 1\right)} \cos\left(u\right) \sin\left(v\right) - \sqrt{3} + 4}{\mathit{\varepsilon} - 1}, \frac{{\left(\mathit{\varepsilon} - 1\right)} \sin\left(u\right) + 1}{\mathit{\varepsilon} - 1}\right) $ & $0$ & $\frac{1}{2} \pi$ & $\pi$ & $\frac{4}{3} \pi$ \\
5 & $ \left(-\frac{{\left(\mathit{\varepsilon} - 1\right)} \cos\left(u\right) \cos\left(v\right) - 1}{\mathit{\varepsilon} - 1}, \frac{{\left(\mathit{\varepsilon} - 1\right)} \cos\left(u\right) \sin\left(v\right) - \sqrt{3} + 4}{\mathit{\varepsilon} - 1}, \frac{{\left(\mathit{\varepsilon} - 1\right)} \sin\left(u\right) + 1}{\mathit{\varepsilon} - 1}\right) $ & $0$ & $\frac{1}{2} \pi$ & $\pi$ & $\frac{4}{3} \pi$ \\
6 & $ \left(\frac{{\left(\mathit{\varepsilon} - 1\right)} \cos\left(v\right) - 1}{\mathit{\varepsilon} - 1}, \frac{{\left({\left(\mathit{\varepsilon} - 1\right)} \sin\left(v\right) - r\right)} \cos\left(u\right) + 2}{\mathit{\varepsilon} - 1}, \frac{{\left({\left(\mathit{\varepsilon} - 1\right)} \sin\left(v\right) - r\right)} \sin\left(u\right) + 1}{\mathit{\varepsilon} - 1}\right) $ & $\pi$ & $2 \pi$ & $\frac{2}{3} \pi$ & $\pi$ \\
7 & $ \left(-\frac{{\left(\mathit{\varepsilon} - 1\right)} \cos\left(v\right) - 1}{\mathit{\varepsilon} - 1}, \frac{{\left({\left(\mathit{\varepsilon} - 1\right)} \sin\left(v\right) - r\right)} \cos\left(u\right) + 2}{\mathit{\varepsilon} - 1}, \frac{{\left({\left(\mathit{\varepsilon} - 1\right)} \sin\left(v\right) - r\right)} \sin\left(u\right) + 1}{\mathit{\varepsilon} - 1}\right) $ & $\pi$ & $2 \pi$ & $\frac{2}{3} \pi$ & $\pi$ \\
8 & $ \left(-\frac{\mathit{\varepsilon}}{\mathit{\varepsilon} - 1}, -\frac{r \cos\left(v\right) - 2}{\mathit{\varepsilon} - 1}, \frac{r u \sin\left(v\right) + 1}{\mathit{\varepsilon} - 1}\right) $ & $0$ & $1$ & $0$ & $\pi$ \\
9 & $ \left(\frac{\mathit{\varepsilon}}{\mathit{\varepsilon} - 1}, -\frac{r \cos\left(v\right) - 2}{\mathit{\varepsilon} - 1}, \frac{r u \sin\left(v\right) + 1}{\mathit{\varepsilon} - 1}\right) $ & $0$ & $1$ & $0$ & $\pi$ \\
10 & $ \left(-\frac{{\left(\mathit{\varepsilon} + 1\right)} \cos\left(v\right)}{\mathit{\varepsilon} - 1}, -\frac{{\left({\left(\mathit{\varepsilon} + 1\right)} \sin\left(v\right) + 2\right)} \cos\left(u\right) - 2}{\mathit{\varepsilon} - 1}, -\frac{{\left({\left(\mathit{\varepsilon} + 1\right)} \sin\left(v\right) + 2\right)} \sin\left(u\right) - 1}{\mathit{\varepsilon} - 1}\right) $ & $\pi$ & $2 \pi$ & $\frac{5}{3} \pi$ & $2 \pi$ \\
11 & $ \left(\frac{{\left({\left(\mathit{\varepsilon} - 1\right)} \sin\left(v\right) - 2\right)} \cos\left(u\right)}{\mathit{\varepsilon} - 1}, \frac{{\left({\left(\mathit{\varepsilon} - 1\right)} \sin\left(v\right) - 2\right)} \sin\left(u\right) + 4}{\mathit{\varepsilon} - 1}, \frac{{\left(\mathit{\varepsilon} - 1\right)} \cos\left(v\right) + 1}{\mathit{\varepsilon} - 1}\right) $ & $0$ & $\frac{1}{3} \pi$ & $-\frac{1}{2} \pi$ & $0$ \\
12 & $ \left(\frac{{\left({\left(\mathit{\varepsilon} - 1\right)} \sin\left(v\right) - 2\right)} \cos\left(u\right)}{\mathit{\varepsilon} - 1}, \frac{{\left({\left(\mathit{\varepsilon} - 1\right)} \sin\left(v\right) - 2\right)} \sin\left(u\right) + 4}{\mathit{\varepsilon} - 1}, \frac{{\left(\mathit{\varepsilon} - 1\right)} \cos\left(v\right) + 1}{\mathit{\varepsilon} - 1}\right) $ & $\frac{2}{3} \pi$ & $\pi$ & $-\frac{1}{2} \pi$ & $0$ \\
13 & $ \left(-\frac{{\left(\mathit{\varepsilon} + 1\right)} \cos\left(v\right)}{\mathit{\varepsilon} - 1}, -\frac{{\left({\left(\mathit{\varepsilon} + 1\right)} \sin\left(v\right) + 2\right)} \cos\left(u\right) - 2}{\mathit{\varepsilon} - 1}, -\frac{{\left({\left(\mathit{\varepsilon} + 1\right)} \sin\left(v\right) + 2\right)} \sin\left(u\right) - 1}{\mathit{\varepsilon} - 1}\right) $ & $\pi$ & $2 \pi$ & $\pi$ & $\frac{4}{3} \pi$ \\
14 & $ \left(-\frac{{\left(\mathit{\varepsilon} + 1\right)} \cos\left(u\right)}{\mathit{\varepsilon} - 1}, -\frac{{\left(\mathit{\varepsilon} + 1\right)} \sin\left(u\right)}{\mathit{\varepsilon} - 1}, -\frac{v}{\mathit{\varepsilon} - 1}\right) $ & $0$ & $\pi$ & $-3$ & $-1$ \\
15 & $ \left(\frac{{\left(\mathit{\varepsilon} - 1\right)} \cos\left(u\right) + 2}{\mathit{\varepsilon} - 1}, -\frac{v}{\mathit{\varepsilon} - 1}, \frac{{\left(\mathit{\varepsilon} - 1\right)} \sin\left(u\right) + 1}{\mathit{\varepsilon} - 1}\right) $ & $\frac{1}{2} \pi$ & $\pi$ & $0$ & $4$ \\
16 & $ \left(\frac{{\left(\mathit{\varepsilon} - 1\right)} \cos\left(u\right) + 2}{\mathit{\varepsilon} - 1}, -\frac{v}{\mathit{\varepsilon} - 1}, \frac{{\left(\mathit{\varepsilon} - 1\right)} \sin\left(u\right) + 1}{\mathit{\varepsilon} - 1}\right) $ & $\frac{1}{2} \pi$ & $\pi$ & $-4$ & $0$ \\
17 & $ \left(\frac{{\left(\mathit{\varepsilon} - 1\right)} \cos\left(u\right) - 2}{\mathit{\varepsilon} - 1}, -\frac{v}{\mathit{\varepsilon} - 1}, \frac{{\left(\mathit{\varepsilon} - 1\right)} \sin\left(u\right) + 1}{\mathit{\varepsilon} - 1}\right) $ & $0$ & $\frac{1}{2} \pi$ & $0$ & $4$ \\
18 & $ \left(\frac{{\left(\mathit{\varepsilon} - 1\right)} \cos\left(u\right) - 2}{\mathit{\varepsilon} - 1}, -\frac{v}{\mathit{\varepsilon} - 1}, \frac{{\left(\mathit{\varepsilon} - 1\right)} \sin\left(u\right) + 1}{\mathit{\varepsilon} - 1}\right) $ & $0$ & $\frac{1}{2} \pi$ & $-4$ & $0$ \\
19 & $ \left(\frac{{\left(\mathit{\varepsilon} - 1\right)} \cos\left(u\right) + 2}{\mathit{\varepsilon} - 1}, -\frac{v}{\mathit{\varepsilon} - 1}, \frac{{\left(\mathit{\varepsilon} - 1\right)} \sin\left(u\right) + 3}{\mathit{\varepsilon} - 1}\right) $ & $\pi$ & $\frac{3}{2} \pi$ & $0$ & $4$ \\
20 & $ \left(\frac{{\left(\mathit{\varepsilon} - 1\right)} \cos\left(u\right) + 2}{\mathit{\varepsilon} - 1}, -\frac{v}{\mathit{\varepsilon} - 1}, \frac{{\left(\mathit{\varepsilon} - 1\right)} \sin\left(u\right) + 3}{\mathit{\varepsilon} - 1}\right) $ & $\pi$ & $\frac{3}{2} \pi$ & $-4$ & $0$ \\
21 & $ \left(\frac{{\left(\mathit{\varepsilon} - 1\right)} \cos\left(u\right) - 2}{\mathit{\varepsilon} - 1}, -\frac{v}{\mathit{\varepsilon} - 1}, \frac{{\left(\mathit{\varepsilon} - 1\right)} \sin\left(u\right) + 3}{\mathit{\varepsilon} - 1}\right) $ & $\frac{3}{2} \pi$ & $2 \pi$ & $0$ & $4$ \\
22 & $ \left(\frac{{\left(\mathit{\varepsilon} - 1\right)} \cos\left(u\right) - 2}{\mathit{\varepsilon} - 1}, -\frac{v}{\mathit{\varepsilon} - 1}, \frac{{\left(\mathit{\varepsilon} - 1\right)} \sin\left(u\right) + 3}{\mathit{\varepsilon} - 1}\right) $ & $\frac{3}{2} \pi$ & $2 \pi$ & $-4$ & $0$ \\
23 & $ \left(\frac{{\left({\left(\mathit{\varepsilon} - 1\right)} \sin\left(u\right) - 2\right)} \cos\left(v\right)}{\mathit{\varepsilon} - 1}, \frac{{\left({\left(\mathit{\varepsilon} - 1\right)} \sin\left(u\right) - 2\right)} \sin\left(v\right) - 4}{\mathit{\varepsilon} - 1}, \frac{{\left(\mathit{\varepsilon} - 1\right)} \cos\left(u\right) + 1}{\mathit{\varepsilon} - 1}\right) $ & $0$ & $\frac{1}{2} \pi$ & $0$ & $\pi$ \\
24 & $ \left(\frac{{\left({\left(\mathit{\varepsilon} - 1\right)} \sin\left(u\right) - 2\right)} \cos\left(v\right)}{\mathit{\varepsilon} - 1}, \frac{{\left({\left(\mathit{\varepsilon} - 1\right)} \sin\left(u\right) - 2\right)} \sin\left(v\right) - 4}{\mathit{\varepsilon} - 1}, \frac{{\left(\mathit{\varepsilon} - 1\right)} \cos\left(u\right) + 3}{\mathit{\varepsilon} - 1}\right) $ & $\frac{1}{2} \pi$ & $\pi$ & $0$ & $\pi$ \\
25 & $ \left(\frac{{\left({\left(\mathit{\varepsilon} - 1\right)} \sin\left(v\right) - 2\right)} \cos\left(u\right)}{\mathit{\varepsilon} - 1}, \frac{{\left({\left(\mathit{\varepsilon} - 1\right)} \sin\left(v\right) - 2\right)} \sin\left(u\right)}{\mathit{\varepsilon} - 1}, \frac{{\left(\mathit{\varepsilon} - 1\right)} \cos\left(v\right) + 3}{\mathit{\varepsilon} - 1}\right) $ & $0$ & $\pi$ & $\pi$ & $\frac{3}{2} \pi$ \\
26 & $ \left(\frac{{\left(\mathit{\varepsilon} - 3\right)} \cos\left(u\right)}{\mathit{\varepsilon} - 1}, \frac{{\left(\mathit{\varepsilon} - 3\right)} \sin\left(u\right) - 4}{\mathit{\varepsilon} - 1}, -\frac{v}{\mathit{\varepsilon} - 1}\right) $ & $0$ & $\pi$ & $-3$ & $-1$ \\
27 & $ \left(\frac{{\left(\mathit{\varepsilon} - 1\right)} \cos\left(u\right) + 2}{\mathit{\varepsilon} - 1}, -\frac{v}{\mathit{\varepsilon} - 1}, \frac{{\left(\mathit{\varepsilon} - 1\right)} \sin\left(u\right) + 3}{\mathit{\varepsilon} - 1}\right) $ & $\frac{3}{2} \pi$ & $2 \pi$ & $-4$ & $0$ \\
28 & $ \left(\frac{{\left(\mathit{\varepsilon} - 1\right)} \cos\left(u\right) - 2}{\mathit{\varepsilon} - 1}, -\frac{v}{\mathit{\varepsilon} - 1}, \frac{{\left(\mathit{\varepsilon} - 1\right)} \sin\left(u\right) + 3}{\mathit{\varepsilon} - 1}\right) $ & $\pi$ & $\frac{3}{2} \pi$ & $-4$ & $0$ \\
29 & $ \left(\frac{{\left(\mathit{\varepsilon} - 1\right)} \cos\left(u\right) + 2}{\mathit{\varepsilon} - 1}, \frac{{\left(\mathit{\varepsilon} - 1\right)} \sin\left(u\right) + 4}{\mathit{\varepsilon} - 1}, -\frac{v}{\mathit{\varepsilon} - 1}\right) $ & $\pi$ & $2 \pi$ & $-3$ & $-1$ \\
30 & $ \left(\frac{{\left(\mathit{\varepsilon} - 1\right)} \cos\left(u\right) - 2}{\mathit{\varepsilon} - 1}, \frac{{\left(\mathit{\varepsilon} - 1\right)} \sin\left(u\right) + 4}{\mathit{\varepsilon} - 1}, -\frac{v}{\mathit{\varepsilon} - 1}\right) $ & $\pi$ & $2 \pi$ & $-3$ & $-1$ \\
31 & $ \left(\frac{{\left(\mathit{\varepsilon} - 1\right)} \cos\left(u\right) \cos\left(v\right) - 2}{\mathit{\varepsilon} - 1}, \frac{{\left(\mathit{\varepsilon} - 1\right)} \cos\left(u\right) \sin\left(v\right) + 4}{\mathit{\varepsilon} - 1}, \frac{{\left(\mathit{\varepsilon} - 1\right)} \sin\left(u\right) + 1}{\mathit{\varepsilon} - 1}\right) $ & $0$ & $\frac{1}{2} \pi$ & $\pi$ & $2 \pi$ \\
32 & $ \left(\frac{{\left(\mathit{\varepsilon} - 1\right)} \cos\left(u\right) \cos\left(v\right) - 2}{\mathit{\varepsilon} - 1}, \frac{{\left(\mathit{\varepsilon} - 1\right)} \cos\left(u\right) \sin\left(v\right) + 4}{\mathit{\varepsilon} - 1}, \frac{{\left(\mathit{\varepsilon} - 1\right)} \sin\left(u\right) + 3}{\mathit{\varepsilon} - 1}\right) $ & $\frac{3}{2} \pi$ & $2 \pi$ & $\pi$ & $2 \pi$ \\
33 & $ \left(\frac{{\left(\mathit{\varepsilon} - 1\right)} \cos\left(u\right) \cos\left(v\right) + 2}{\mathit{\varepsilon} - 1}, \frac{{\left(\mathit{\varepsilon} - 1\right)} \cos\left(u\right) \sin\left(v\right) + 4}{\mathit{\varepsilon} - 1}, \frac{{\left(\mathit{\varepsilon} - 1\right)} \sin\left(u\right) + 1}{\mathit{\varepsilon} - 1}\right) $ & $0$ & $\frac{1}{2} \pi$ & $\pi$ & $2 \pi$ \\
34 & $ \left(\frac{{\left(\mathit{\varepsilon} - 1\right)} \cos\left(u\right) \cos\left(v\right) + 2}{\mathit{\varepsilon} - 1}, \frac{{\left(\mathit{\varepsilon} - 1\right)} \cos\left(u\right) \sin\left(v\right) + 4}{\mathit{\varepsilon} - 1}, \frac{{\left(\mathit{\varepsilon} - 1\right)} \sin\left(u\right) + 3}{\mathit{\varepsilon} - 1}\right) $ & $\frac{3}{2} \pi$ & $2 \pi$ & $\pi$ & $2 \pi$ \\
35 & $ \left(-\frac{\mathit{\varepsilon} - 3}{\mathit{\varepsilon} - 1}, -\frac{u}{\mathit{\varepsilon} - 1}, -\frac{v}{\mathit{\varepsilon} - 1}\right) $ & $0$ & $4$ & $-3$ & $-1$ \\
36 & $ \left(-\frac{\mathit{\varepsilon} - 3}{\mathit{\varepsilon} - 1}, -\frac{u}{\mathit{\varepsilon} - 1}, -\frac{v}{\mathit{\varepsilon} - 1}\right) $ & $-4$ & $0$ & $-3$ & $-1$ \\
37 & $ \left(\frac{\mathit{\varepsilon} - 3}{\mathit{\varepsilon} - 1}, -\frac{u}{\mathit{\varepsilon} - 1}, -\frac{v}{\mathit{\varepsilon} - 1}\right) $ & $0$ & $4$ & $-3$ & $-1$ \\
38 & $ \left(\frac{\mathit{\varepsilon} - 3}{\mathit{\varepsilon} - 1}, -\frac{u}{\mathit{\varepsilon} - 1}, -\frac{v}{\mathit{\varepsilon} - 1}\right) $ & $-4$ & $0$ & $-3$ & $-1$ \\
39 & $ \left(-\frac{2 \, \cos\left(u\right)}{\mathit{\varepsilon} - 1}, -\frac{v + 2 \, \sin\left(u\right)}{\mathit{\varepsilon} - 1}, -\frac{\mathit{\varepsilon} - 4}{\mathit{\varepsilon} - 1}\right) $ & $0$ & $\pi$ & $0$ & $4$ \\
40 & $ \left(-\frac{2 \, \cos\left(u\right)}{\mathit{\varepsilon} - 1}, -\frac{2 \, {\left(v {\left(\sin\left(u\right) - 1\right)} - 1\right)}}{\mathit{\varepsilon} - 1}, \frac{\mathit{\varepsilon}}{\mathit{\varepsilon} - 1}\right) $ & $0$ & $\frac{1}{3} \pi$ & $-1$ & $1$ \\
41 & $ \left(-\frac{2 \, \cos\left(u\right)}{\mathit{\varepsilon} - 1}, -\frac{2 \, {\left(v {\left(\sin\left(u\right) - 1\right)} - 1\right)}}{\mathit{\varepsilon} - 1}, \frac{\mathit{\varepsilon}}{\mathit{\varepsilon} - 1}\right) $ & $\frac{2}{3} \pi$ & $\pi$ & $-1$ & $1$ \\
42 & $ \left(\frac{\mathit{\varepsilon} + 1}{\mathit{\varepsilon} - 1}, -\frac{2 \, {\left(\cos\left(u\right) - 1\right)}}{\mathit{\varepsilon} - 1}, \frac{2 \, v {\left(\sin\left(u\right) - 1\right)} + 3}{\mathit{\varepsilon} - 1}\right) $ & $0$ & $\pi$ & $0$ & $1$ \\
43 & $ \left(-\frac{\mathit{\varepsilon} + 1}{\mathit{\varepsilon} - 1}, -\frac{2 \, {\left(\cos\left(u\right) - 1\right)}}{\mathit{\varepsilon} - 1}, \frac{2 \, v {\left(\sin\left(u\right) - 1\right)} + 3}{\mathit{\varepsilon} - 1}\right) $ & $0$ & $\pi$ & $0$ & $1$ \\
44 & $ \left(-\frac{2 \, \cos\left(u\right)}{\mathit{\varepsilon} - 1}, -\frac{2 \, {\left(v {\left(\sin\left(u\right) - 1\right)} + 1\right)}}{\mathit{\varepsilon} - 1}, \frac{\mathit{\varepsilon}}{\mathit{\varepsilon} - 1}\right) $ & $0$ & $\pi$ & $-1$ & $1$ \\
45 & $ \left(-\frac{{\left({\left(\mathit{\varepsilon} - 1\right)} \sin\left(v\right) - 2\right)} \cos\left(u\right)}{\mathit{\varepsilon} - 1}, \frac{{\left({\left(\mathit{\varepsilon} - 1\right)} \sin\left(v\right) - 2\right)} \sin\left(u\right)}{\mathit{\varepsilon} - 1}, \frac{{\left(\mathit{\varepsilon} - 1\right)} \cos\left(v\right) + 1}{\mathit{\varepsilon} - 1}\right) $ & $-\frac{1}{3} \pi$ & $0$ & $-\frac{1}{2} \pi$ & $0$ \\
46 & $ \left(-\frac{{\left({\left(\mathit{\varepsilon} - 1\right)} \sin\left(v\right) - 2\right)} \cos\left(u\right)}{\mathit{\varepsilon} - 1}, \frac{{\left({\left(\mathit{\varepsilon} - 1\right)} \sin\left(v\right) - 2\right)} \sin\left(u\right)}{\mathit{\varepsilon} - 1}, \frac{{\left(\mathit{\varepsilon} - 1\right)} \cos\left(v\right) + 1}{\mathit{\varepsilon} - 1}\right) $ & $0$ & $\pi$ & $-\frac{1}{2} \pi$ & $0$ \\
47 & $ \left(-\frac{{\left({\left(\mathit{\varepsilon} - 1\right)} \sin\left(v\right) - 2\right)} \cos\left(u\right)}{\mathit{\varepsilon} - 1}, \frac{{\left({\left(\mathit{\varepsilon} - 1\right)} \sin\left(v\right) - 2\right)} \sin\left(u\right)}{\mathit{\varepsilon} - 1}, \frac{{\left(\mathit{\varepsilon} - 1\right)} \cos\left(v\right) + 1}{\mathit{\varepsilon} - 1}\right) $ & $\pi$ & $\frac{4}{3} \pi$ & $-\frac{1}{2} \pi$ & $0$ \\
48 & $ \left(-\frac{{\left({\left(\mathit{\varepsilon} + 1\right)} \sin\left(v\right) + 2\right)} \cos\left(u\right)}{\mathit{\varepsilon} - 1}, -\frac{{\left({\left(\mathit{\varepsilon} + 1\right)} \sin\left(v\right) + 2\right)} \sin\left(u\right)}{\mathit{\varepsilon} - 1}, -\frac{{\left(\mathit{\varepsilon} + 1\right)} \cos\left(v\right) + 1}{\mathit{\varepsilon} - 1}\right) $ & $0$ & $\pi$ & $\frac{4}{3} \pi$ & $\frac{3}{2} \pi$ \\
49 & $ \left(-\frac{{\left({\left(\mathit{\varepsilon} + 1\right)} \sin\left(v\right) + 2\right)} \cos\left(u\right)}{\mathit{\varepsilon} - 1}, -\frac{{\left({\left(\mathit{\varepsilon} + 1\right)} \sin\left(v\right) + 2\right)} \sin\left(u\right)}{\mathit{\varepsilon} - 1}, -\frac{{\left(\mathit{\varepsilon} + 1\right)} \cos\left(v\right) + 1}{\mathit{\varepsilon} - 1}\right) $ & $\pi$ & $2 \pi$ & $\frac{4}{3} \pi$ & $\frac{3}{2} \pi$ \\
50 & $ \left(-\frac{{\left({\left(\mathit{\varepsilon} + 1\right)} \sin\left(v\right) + 2\right)} \cos\left(u\right)}{\mathit{\varepsilon} - 1}, -\frac{{\left({\left(\mathit{\varepsilon} + 1\right)} \sin\left(v\right) + 2\right)} \sin\left(u\right) + 4}{\mathit{\varepsilon} - 1}, -\frac{{\left(\mathit{\varepsilon} + 1\right)} \cos\left(v\right) + 1}{\mathit{\varepsilon} - 1}\right) $ & $\pi$ & $2 \pi$ & $\pi$ & $\frac{3}{2} \pi$ \\
51 & $ \left(-\frac{{\left({\left(\mathit{\varepsilon} + 1\right)} \sin\left(v\right) + 2\right)} \cos\left(u\right)}{\mathit{\varepsilon} - 1}, -\frac{{\left({\left(\mathit{\varepsilon} + 1\right)} \sin\left(v\right) + 2\right)} \sin\left(u\right) + 4}{\mathit{\varepsilon} - 1}, -\frac{{\left(\mathit{\varepsilon} + 1\right)} \cos\left(v\right) + 1}{\mathit{\varepsilon} - 1}\right) $ & $0$ & $\pi$ & $\pi$ & $\frac{3}{2} \pi$ \\
52 & $ \left(\cos\left(v\right), \frac{{\left({\left(\mathit{\varepsilon} - 1\right)} \sin\left(v\right) - 2\right)} \cos\left(u\right) - 2}{\mathit{\varepsilon} - 1}, \frac{{\left({\left(\mathit{\varepsilon} - 1\right)} \sin\left(v\right) - 2\right)} \sin\left(u\right) - 1}{\mathit{\varepsilon} - 1}\right) $ & $0$ & $\pi$ & $0$ & $\pi$ \\
53 & $ \left(\cos\left(v\right), \frac{{\left({\left(\mathit{\varepsilon} - 1\right)} \sin\left(v\right) - 2\right)} \cos\left(u\right) - 2}{\mathit{\varepsilon} - 1}, \frac{{\left({\left(\mathit{\varepsilon} - 1\right)} \sin\left(v\right) - 2\right)} \sin\left(u\right) - 1}{\mathit{\varepsilon} - 1}\right) $ & $0$ & $\pi$ & $\pi$ & $2 \pi$ \\
54 & $ \left(-\frac{{\left({\left(\mathit{\varepsilon} - 1\right)} \sin\left(v\right) - r\right)} \cos\left(u\right)}{\mathit{\varepsilon} - 1}, -\frac{{\left({\left(\mathit{\varepsilon} - 1\right)} \sin\left(v\right) - r\right)} \sin\left(u\right)}{\mathit{\varepsilon} - 1}, -\cos\left(v\right)\right) $ & $\pi$ & $2 \pi$ & $\frac{1}{3} \pi$ & $\frac{2}{3} \pi$ \\
55 & $ \left(-\frac{{\left({\left(\mathit{\varepsilon} - 1\right)} \sin\left(v\right) - r\right)} \cos\left(u\right)}{\mathit{\varepsilon} - 1}, -\frac{{\left({\left(\mathit{\varepsilon} - 1\right)} \sin\left(v\right) - r\right)} \sin\left(u\right)}{\mathit{\varepsilon} - 1}, -\cos\left(v\right)\right) $ & $0$ & $\pi$ & $\frac{1}{3} \pi$ & $\frac{2}{3} \pi$ \\
56 & $ \left(\cos\left(v\right), \frac{{\left({\left(\mathit{\varepsilon} - 1\right)} \sin\left(v\right) - 2\right)} \cos\left(u\right) + 2}{\mathit{\varepsilon} - 1}, \frac{{\left({\left(\mathit{\varepsilon} - 1\right)} \sin\left(v\right) - 2\right)} \sin\left(u\right) + 1}{\mathit{\varepsilon} - 1}\right) $ & $\pi$ & $2 \pi$ & $\pi$ & $2 \pi$ \\
57 & $ \left(-\frac{{\left({\left(\mathit{\varepsilon} + 1\right)} \sin\left(v\right) + 2\right)} \cos\left(u\right)}{\mathit{\varepsilon} - 1}, -\frac{{\left({\left(\mathit{\varepsilon} + 1\right)} \sin\left(v\right) + 2\right)} \sin\left(u\right) - 4}{\mathit{\varepsilon} - 1}, -\frac{{\left(\mathit{\varepsilon} + 1\right)} \cos\left(v\right) - 1}{\mathit{\varepsilon} - 1}\right) $ & $0$ & $\pi$ & $-\frac{1}{2} \pi$ & $0$ \\
58 & $ \left(\frac{{\left({\left(\mathit{\varepsilon} + 1\right)} \sin\left(v\right) + 2\right)} \cos\left(u\right)}{\mathit{\varepsilon} - 1}, -\frac{{\left({\left(\mathit{\varepsilon} + 1\right)} \sin\left(v\right) + 2\right)} \sin\left(u\right)}{\mathit{\varepsilon} - 1}, -\frac{{\left(\mathit{\varepsilon} + 1\right)} \cos\left(v\right) - 1}{\mathit{\varepsilon} - 1}\right) $ & $\pi$ & $2 \pi$ & $-\frac{1}{2} \pi$ & $-\frac{1}{3} \pi$ \\
59 & $ \left(\frac{{\left({\left(\mathit{\varepsilon} + 1\right)} \sin\left(v\right) + 2\right)} \cos\left(u\right)}{\mathit{\varepsilon} - 1}, -\frac{{\left({\left(\mathit{\varepsilon} + 1\right)} \sin\left(v\right) + 2\right)} \sin\left(u\right)}{\mathit{\varepsilon} - 1}, -\frac{{\left(\mathit{\varepsilon} + 1\right)} \cos\left(v\right) - 1}{\mathit{\varepsilon} - 1}\right) $ & $0$ & $\pi$ & $-\frac{1}{2} \pi$ & $-\frac{1}{3} \pi$ \\
60 & $ \left(\cos\left(u\right), \sin\left(u\right), -\frac{v}{\mathit{\varepsilon} - 1}\right) $ & $0$ & $\pi$ & $-3$ & $-1$ \\
61 & $ \left(-1, -\frac{2 \, {\left(\cos\left(u\right) - 1\right)}}{\mathit{\varepsilon} - 1}, \frac{2 \, v {\left(\sin\left(u\right) - 1\right)} + 3}{\mathit{\varepsilon} - 1}\right) $ & $0$ & $\pi$ & $0$ & $1$ \\
62 & $ \left(1, -\frac{2 \, {\left(\cos\left(u\right) - 1\right)}}{\mathit{\varepsilon} - 1}, \frac{2 \, v {\left(\sin\left(u\right) - 1\right)} + 3}{\mathit{\varepsilon} - 1}\right) $ & $0$ & $\pi$ & $0$ & $1$ \\
63 & $ \left(-\frac{{\left({\left(\mathit{\varepsilon} + 1\right)} \sin\left(v\right) + 2\right)} \cos\left(u\right)}{\mathit{\varepsilon} - 1}, -\frac{{\left({\left(\mathit{\varepsilon} + 1\right)} \sin\left(v\right) + 2\right)} \sin\left(u\right)}{\mathit{\varepsilon} - 1}, -\frac{{\left(\mathit{\varepsilon} + 1\right)} \cos\left(v\right) - 3}{\mathit{\varepsilon} - 1}\right) $ & $0$ & $\pi$ & $\pi$ & $\frac{3}{2} \pi$ \\
64 & $ \left(-\frac{2 \, \cos\left(u\right)}{\mathit{\varepsilon} - 1}, -\frac{2 \, {\left(v {\left(\sin\left(u\right) - 1\right)} + 1\right)}}{\mathit{\varepsilon} - 1}, -\frac{\mathit{\varepsilon}}{\mathit{\varepsilon} - 1}\right) $ & $0$ & $\frac{1}{3} \pi$ & $-1$ & $1$ \\
65 & $ \left(-\frac{2 \, \cos\left(u\right)}{\mathit{\varepsilon} - 1}, -\frac{2 \, {\left(v {\left(\sin\left(u\right) - 1\right)} + 1\right)}}{\mathit{\varepsilon} - 1}, -\frac{\mathit{\varepsilon}}{\mathit{\varepsilon} - 1}\right) $ & $\frac{2}{3} \pi$ & $\pi$ & $-1$ & $1$ \\
66 & $ \left(-\frac{2 \, \cos\left(u\right)}{\mathit{\varepsilon} - 1}, -\frac{2 \, {\left(v {\left(\sin\left(u\right) - 1\right)} - 1\right)}}{\mathit{\varepsilon} - 1}, -\frac{\mathit{\varepsilon}}{\mathit{\varepsilon} - 1}\right) $ & $0$ & $\pi$ & $-1$ & $1$ \\
67 & $ \left(-\frac{{\left(\mathit{\varepsilon} + 1\right)} \cos\left(u\right) \cos\left(v\right) - 2}{\mathit{\varepsilon} - 1}, -\frac{{\left(\mathit{\varepsilon} + 1\right)} \cos\left(u\right) \sin\left(v\right) - 4}{\mathit{\varepsilon} - 1}, -\frac{{\left(\mathit{\varepsilon} + 1\right)} \sin\left(u\right) - 1}{\mathit{\varepsilon} - 1}\right) $ & $0$ & $\frac{1}{2} \pi$ & $\pi$ & $2 \pi$ \\
68 & $ \left(-\frac{{\left(\mathit{\varepsilon} + 1\right)} \cos\left(u\right) - 2}{\mathit{\varepsilon} - 1}, -\frac{v}{\mathit{\varepsilon} - 1}, -\frac{{\left(\mathit{\varepsilon} + 1\right)} \sin\left(u\right) - 1}{\mathit{\varepsilon} - 1}\right) $ & $\frac{1}{2} \pi$ & $\pi$ & $-4$ & $0$ \\
69 & $ \left(-\frac{{\left(\mathit{\varepsilon} + 1\right)} \cos\left(u\right) - 2}{\mathit{\varepsilon} - 1}, -\frac{v}{\mathit{\varepsilon} - 1}, -\frac{{\left(\mathit{\varepsilon} + 1\right)} \sin\left(u\right) - 1}{\mathit{\varepsilon} - 1}\right) $ & $\frac{1}{2} \pi$ & $\pi$ & $0$ & $4$ \\
70 & $ \left(-\frac{{\left({\left(\mathit{\varepsilon} + 1\right)} \sin\left(u\right) + 2\right)} \cos\left(v\right)}{\mathit{\varepsilon} - 1}, -\frac{{\left({\left(\mathit{\varepsilon} + 1\right)} \sin\left(u\right) + 2\right)} \sin\left(v\right) + 4}{\mathit{\varepsilon} - 1}, -\frac{{\left(\mathit{\varepsilon} + 1\right)} \cos\left(u\right) - 1}{\mathit{\varepsilon} - 1}\right) $ & $0$ & $\frac{1}{2} \pi$ & $0$ & $\pi$ \\
71 & $ \left(-\frac{{\left(\mathit{\varepsilon} + 1\right)} \cos\left(u\right) + 2}{\mathit{\varepsilon} - 1}, -\frac{v}{\mathit{\varepsilon} - 1}, -\frac{{\left(\mathit{\varepsilon} + 1\right)} \sin\left(u\right) - 1}{\mathit{\varepsilon} - 1}\right) $ & $0$ & $\frac{1}{2} \pi$ & $-4$ & $0$ \\
72 & $ \left(-\frac{{\left(\mathit{\varepsilon} + 1\right)} \cos\left(u\right) + 2}{\mathit{\varepsilon} - 1}, -\frac{v}{\mathit{\varepsilon} - 1}, -\frac{{\left(\mathit{\varepsilon} + 1\right)} \sin\left(u\right) - 1}{\mathit{\varepsilon} - 1}\right) $ & $0$ & $\frac{1}{2} \pi$ & $0$ & $4$ \\
73 & $ \left(-\frac{{\left(\mathit{\varepsilon} + 1\right)} \cos\left(u\right) \cos\left(v\right) + 2}{\mathit{\varepsilon} - 1}, -\frac{{\left(\mathit{\varepsilon} + 1\right)} \cos\left(u\right) \sin\left(v\right) - 4}{\mathit{\varepsilon} - 1}, -\frac{{\left(\mathit{\varepsilon} + 1\right)} \sin\left(u\right) - 1}{\mathit{\varepsilon} - 1}\right) $ & $0$ & $\frac{1}{2} \pi$ & $\pi$ & $2 \pi$ \\
74 & $ \left(-\frac{\mathit{\varepsilon} + 3}{\mathit{\varepsilon} - 1}, -\frac{u}{\mathit{\varepsilon} - 1}, -\frac{v}{\mathit{\varepsilon} - 1}\right) $ & $0$ & $4$ & $-3$ & $-1$ \\
75 & $ \left(-\frac{\mathit{\varepsilon} + 3}{\mathit{\varepsilon} - 1}, -\frac{u}{\mathit{\varepsilon} - 1}, -\frac{v}{\mathit{\varepsilon} - 1}\right) $ & $-4$ & $0$ & $-3$ & $-1$ \\
76 & $ \left(-\frac{{\left(\mathit{\varepsilon} + 1\right)} \cos\left(u\right) + 2}{\mathit{\varepsilon} - 1}, -\frac{{\left(\mathit{\varepsilon} + 1\right)} \sin\left(u\right) - 4}{\mathit{\varepsilon} - 1}, -\frac{v}{\mathit{\varepsilon} - 1}\right) $ & $\pi$ & $2 \pi$ & $-3$ & $-1$ \\
77 & $ \left(-\frac{{\left(\mathit{\varepsilon} + 1\right)} \cos\left(u\right) - 2}{\mathit{\varepsilon} - 1}, -\frac{{\left(\mathit{\varepsilon} + 1\right)} \sin\left(u\right) - 4}{\mathit{\varepsilon} - 1}, -\frac{v}{\mathit{\varepsilon} - 1}\right) $ & $\pi$ & $2 \pi$ & $-3$ & $-1$ \\
78 & $ \left(\frac{\mathit{\varepsilon} + 3}{\mathit{\varepsilon} - 1}, -\frac{u}{\mathit{\varepsilon} - 1}, -\frac{v}{\mathit{\varepsilon} - 1}\right) $ & $0$ & $4$ & $-3$ & $-1$ \\
79 & $ \left(\frac{\mathit{\varepsilon} + 3}{\mathit{\varepsilon} - 1}, -\frac{u}{\mathit{\varepsilon} - 1}, -\frac{v}{\mathit{\varepsilon} - 1}\right) $ & $-4$ & $0$ & $-3$ & $-1$ \\
80 & $ \left(-\frac{{\left(\mathit{\varepsilon} + 3\right)} \cos\left(u\right)}{\mathit{\varepsilon} - 1}, -\frac{{\left(\mathit{\varepsilon} + 3\right)} \sin\left(u\right) + 4}{\mathit{\varepsilon} - 1}, -\frac{v}{\mathit{\varepsilon} - 1}\right) $ & $0$ & $\pi$ & $-3$ & $-1$ \\
81 & $ \left(-\frac{{\left(\mathit{\varepsilon} + 1\right)} \cos\left(u\right) - 2}{\mathit{\varepsilon} - 1}, -\frac{v}{\mathit{\varepsilon} - 1}, -\frac{{\left(\mathit{\varepsilon} + 1\right)} \sin\left(u\right) - 3}{\mathit{\varepsilon} - 1}\right) $ & $\frac{3}{2} \pi$ & $2 \pi$ & $-4$ & $0$ \\
82 & $ \left(-\frac{{\left(\mathit{\varepsilon} + 1\right)} \cos\left(u\right) \cos\left(v\right) - 2}{\mathit{\varepsilon} - 1}, -\frac{{\left(\mathit{\varepsilon} + 1\right)} \cos\left(u\right) \sin\left(v\right) - 4}{\mathit{\varepsilon} - 1}, -\frac{{\left(\mathit{\varepsilon} + 1\right)} \sin\left(u\right) - 3}{\mathit{\varepsilon} - 1}\right) $ & $\frac{3}{2} \pi$ & $2 \pi$ & $\pi$ & $2 \pi$ \\
83 & $ \left(-\frac{{\left(\mathit{\varepsilon} + 1\right)} \cos\left(u\right) - 2}{\mathit{\varepsilon} - 1}, -\frac{v}{\mathit{\varepsilon} - 1}, -\frac{{\left(\mathit{\varepsilon} + 1\right)} \sin\left(u\right) - 3}{\mathit{\varepsilon} - 1}\right) $ & $\pi$ & $\frac{3}{2} \pi$ & $0$ & $4$ \\
84 & $ \left(-\frac{{\left(\mathit{\varepsilon} + 1\right)} \cos\left(u\right) - 2}{\mathit{\varepsilon} - 1}, -\frac{v}{\mathit{\varepsilon} - 1}, -\frac{{\left(\mathit{\varepsilon} + 1\right)} \sin\left(u\right) - 3}{\mathit{\varepsilon} - 1}\right) $ & $\pi$ & $\frac{3}{2} \pi$ & $-4$ & $0$ \\
85 & $ \left(-\frac{{\left({\left(\mathit{\varepsilon} + 1\right)} \sin\left(u\right) + 2\right)} \cos\left(v\right)}{\mathit{\varepsilon} - 1}, -\frac{{\left({\left(\mathit{\varepsilon} + 1\right)} \sin\left(u\right) + 2\right)} \sin\left(v\right) + 4}{\mathit{\varepsilon} - 1}, -\frac{{\left(\mathit{\varepsilon} + 1\right)} \cos\left(u\right) - 3}{\mathit{\varepsilon} - 1}\right) $ & $\frac{1}{2} \pi$ & $\pi$ & $0$ & $\pi$ \\
86 & $ \left(-\frac{{\left(\mathit{\varepsilon} + 1\right)} \cos\left(u\right) + 2}{\mathit{\varepsilon} - 1}, -\frac{v}{\mathit{\varepsilon} - 1}, -\frac{{\left(\mathit{\varepsilon} + 1\right)} \sin\left(u\right) - 3}{\mathit{\varepsilon} - 1}\right) $ & $\frac{3}{2} \pi$ & $2 \pi$ & $0$ & $4$ \\
87 & $ \left(-\frac{{\left(\mathit{\varepsilon} + 1\right)} \cos\left(u\right) + 2}{\mathit{\varepsilon} - 1}, -\frac{v}{\mathit{\varepsilon} - 1}, -\frac{{\left(\mathit{\varepsilon} + 1\right)} \sin\left(u\right) - 3}{\mathit{\varepsilon} - 1}\right) $ & $\frac{3}{2} \pi$ & $2 \pi$ & $-4$ & $0$ \\
88 & $ \left(-\frac{{\left(\mathit{\varepsilon} + 1\right)} \cos\left(u\right) \cos\left(v\right) + 2}{\mathit{\varepsilon} - 1}, -\frac{{\left(\mathit{\varepsilon} + 1\right)} \cos\left(u\right) \sin\left(v\right) - 4}{\mathit{\varepsilon} - 1}, -\frac{{\left(\mathit{\varepsilon} + 1\right)} \sin\left(u\right) - 3}{\mathit{\varepsilon} - 1}\right) $ & $\frac{3}{2} \pi$ & $2 \pi$ & $\pi$ & $2 \pi$ \\
89 & $ \left(-\frac{{\left(\mathit{\varepsilon} + 1\right)} \cos\left(u\right) + 2}{\mathit{\varepsilon} - 1}, -\frac{v}{\mathit{\varepsilon} - 1}, -\frac{{\left(\mathit{\varepsilon} + 1\right)} \sin\left(u\right) - 3}{\mathit{\varepsilon} - 1}\right) $ & $\pi$ & $\frac{3}{2} \pi$ & $-4$ & $0$ \\
90 & $ \left(-\frac{2 \, \cos\left(u\right)}{\mathit{\varepsilon} - 1}, -\frac{v + 2 \, \sin\left(u\right)}{\mathit{\varepsilon} - 1}, \frac{\mathit{\varepsilon} + 4}{\mathit{\varepsilon} - 1}\right) $ & $0$ & $\pi$ & $0$ & $4$ \\
91 & $ \left(\frac{{\left({\left(\mathit{\varepsilon} - 1\right)} \sin\left(v\right) - 2\right)} \cos\left(u\right)}{\mathit{\varepsilon} - 1}, \frac{{\left({\left(\mathit{\varepsilon} - 1\right)} \sin\left(v\right) - 2\right)} \sin\left(u\right)}{\mathit{\varepsilon} - 1}, \frac{{\left(\mathit{\varepsilon} - 1\right)} \cos\left(v\right) - 1}{\mathit{\varepsilon} - 1}\right) $ & $\frac{2}{3} \pi$ & $\pi$ & $\pi$ & $\frac{3}{2} \pi$ \\
92 & $ \left(\frac{{\left({\left(\mathit{\varepsilon} - 1\right)} \sin\left(v\right) - 2\right)} \cos\left(u\right)}{\mathit{\varepsilon} - 1}, \frac{{\left({\left(\mathit{\varepsilon} - 1\right)} \sin\left(v\right) - 2\right)} \sin\left(u\right)}{\mathit{\varepsilon} - 1}, \frac{{\left(\mathit{\varepsilon} - 1\right)} \cos\left(v\right) - 1}{\mathit{\varepsilon} - 1}\right) $ & $\pi$ & $2 \pi$ & $\pi$ & $\frac{3}{2} \pi$ \\
93 & $ \left(\frac{{\left({\left(\mathit{\varepsilon} - 1\right)} \sin\left(v\right) - 2\right)} \cos\left(u\right)}{\mathit{\varepsilon} - 1}, \frac{{\left({\left(\mathit{\varepsilon} - 1\right)} \sin\left(v\right) - 2\right)} \sin\left(u\right)}{\mathit{\varepsilon} - 1}, \frac{{\left(\mathit{\varepsilon} - 1\right)} \cos\left(v\right) - 1}{\mathit{\varepsilon} - 1}\right) $ & $0$ & $\frac{1}{3} \pi$ & $\pi$ & $\frac{3}{2} \pi$ \\
94 & $ \left(-\frac{{\left(\mathit{\varepsilon} + 1\right)} \cos\left(v\right)}{\mathit{\varepsilon} - 1}, -\frac{{\left({\left(\mathit{\varepsilon} + 1\right)} \sin\left(v\right) + 2\right)} \cos\left(u\right) + 2}{\mathit{\varepsilon} - 1}, -\frac{{\left({\left(\mathit{\varepsilon} + 1\right)} \sin\left(v\right) + 2\right)} \sin\left(u\right) + 1}{\mathit{\varepsilon} - 1}\right) $ & $0$ & $\pi$ & $-\frac{1}{3} \pi$ & $0$ \\
95 & $ \left(-\frac{{\left(\mathit{\varepsilon} + 1\right)} \cos\left(v\right)}{\mathit{\varepsilon} - 1}, -\frac{{\left({\left(\mathit{\varepsilon} + 1\right)} \sin\left(v\right) + 2\right)} \cos\left(u\right) + 2}{\mathit{\varepsilon} - 1}, -\frac{{\left({\left(\mathit{\varepsilon} + 1\right)} \sin\left(v\right) + 2\right)} \sin\left(u\right) + 1}{\mathit{\varepsilon} - 1}\right) $ & $0$ & $\pi$ & $0$ & $\pi$ \\
96 & $ \left(-\frac{{\left(\mathit{\varepsilon} + 1\right)} \cos\left(v\right)}{\mathit{\varepsilon} - 1}, -\frac{{\left({\left(\mathit{\varepsilon} + 1\right)} \sin\left(v\right) + 2\right)} \cos\left(u\right) + 2}{\mathit{\varepsilon} - 1}, -\frac{{\left({\left(\mathit{\varepsilon} + 1\right)} \sin\left(v\right) + 2\right)} \sin\left(u\right) + 1}{\mathit{\varepsilon} - 1}\right) $ & $0$ & $\pi$ & $\pi$ & $\frac{4}{3} \pi$ \\
97 & $ \left(\frac{{\left({\left(\mathit{\varepsilon} - 1\right)} \sin\left(v\right) - 2\right)} \cos\left(u\right)}{\mathit{\varepsilon} - 1}, \frac{{\left({\left(\mathit{\varepsilon} - 1\right)} \sin\left(v\right) - 2\right)} \sin\left(u\right) - 4}{\mathit{\varepsilon} - 1}, \frac{{\left(\mathit{\varepsilon} - 1\right)} \cos\left(v\right) - 1}{\mathit{\varepsilon} - 1}\right) $ & $-\frac{1}{3} \pi$ & $0$ & $\pi$ & $\frac{3}{2} \pi$ \\
98 & $ \left(\frac{{\left({\left(\mathit{\varepsilon} - 1\right)} \sin\left(v\right) - 2\right)} \cos\left(u\right)}{\mathit{\varepsilon} - 1}, \frac{{\left({\left(\mathit{\varepsilon} - 1\right)} \sin\left(v\right) - 2\right)} \sin\left(u\right) - 4}{\mathit{\varepsilon} - 1}, \frac{{\left(\mathit{\varepsilon} - 1\right)} \cos\left(v\right) - 1}{\mathit{\varepsilon} - 1}\right) $ & $0$ & $\pi$ & $\pi$ & $\frac{3}{2} \pi$ \\
99 & $ \left(\frac{{\left({\left(\mathit{\varepsilon} - 1\right)} \sin\left(v\right) - 2\right)} \cos\left(u\right)}{\mathit{\varepsilon} - 1}, \frac{{\left({\left(\mathit{\varepsilon} - 1\right)} \sin\left(v\right) - 2\right)} \sin\left(u\right) - 4}{\mathit{\varepsilon} - 1}, \frac{{\left(\mathit{\varepsilon} - 1\right)} \cos\left(v\right) - 1}{\mathit{\varepsilon} - 1}\right) $ & $\pi$ & $\frac{4}{3} \pi$ & $\pi$ & $\frac{3}{2} \pi$ \\
100 & $ \left(\frac{{\left(\mathit{\varepsilon} - 1\right)} \cos\left(v\right) - 1}{\mathit{\varepsilon} - 1}, \frac{{\left({\left(\mathit{\varepsilon} - 1\right)} \sin\left(v\right) - r\right)} \cos\left(u\right) - 2}{\mathit{\varepsilon} - 1}, \frac{{\left({\left(\mathit{\varepsilon} - 1\right)} \sin\left(v\right) - r\right)} \sin\left(u\right) - 1}{\mathit{\varepsilon} - 1}\right) $ & $0$ & $\pi$ & $\frac{2}{3} \pi$ & $\pi$ \\
101 & $ \left(\frac{{\left(\mathit{\varepsilon} - 1\right)} \cos\left(v\right) - 1}{\mathit{\varepsilon} - 1}, \frac{u - 2}{\mathit{\varepsilon} - 1}, -\frac{{\left(\mathit{\varepsilon} - 1\right)} \sin\left(v\right) + 1}{\mathit{\varepsilon} - 1}\right) $ & $-r$ & $r$ & $\frac{1}{2} \pi$ & $\pi$ \\
102 & $ \left(\frac{{\left(\mathit{\varepsilon} - 1\right)} \cos\left(u\right) \cos\left(v\right) - 1}{\mathit{\varepsilon} - 1}, -\frac{{\left(\mathit{\varepsilon} - 1\right)} \cos\left(u\right) \sin\left(v\right) - \sqrt{3} + 4}{\mathit{\varepsilon} - 1}, -\frac{{\left(\mathit{\varepsilon} - 1\right)} \sin\left(u\right) + 1}{\mathit{\varepsilon} - 1}\right) $ & $0$ & $\frac{1}{2} \pi$ & $\pi$ & $\frac{4}{3} \pi$ \\
103 & $ \left(\frac{{\left(\mathit{\varepsilon} - 1\right)} \cos\left(u\right) \cos\left(v\right) - 1}{\mathit{\varepsilon} - 1}, -\frac{{\left(\mathit{\varepsilon} - 1\right)} \cos\left(u\right) \sin\left(v\right) + \sqrt{3}}{\mathit{\varepsilon} - 1}, -\frac{{\left(\mathit{\varepsilon} - 1\right)} \sin\left(u\right) + 1}{\mathit{\varepsilon} - 1}\right) $ & $0$ & $\frac{1}{2} \pi$ & $\frac{2}{3} \pi$ & $\pi$ \\
104 & $ \left(-\frac{{\left(\mathit{\varepsilon} - 1\right)} \cos\left(u\right) \cos\left(v\right) - 1}{\mathit{\varepsilon} - 1}, -\frac{{\left(\mathit{\varepsilon} - 1\right)} \cos\left(u\right) \sin\left(v\right) + \sqrt{3}}{\mathit{\varepsilon} - 1}, -\frac{{\left(\mathit{\varepsilon} - 1\right)} \sin\left(u\right) + 1}{\mathit{\varepsilon} - 1}\right) $ & $0$ & $\frac{1}{2} \pi$ & $\frac{2}{3} \pi$ & $\pi$ \\
105 & $ \left(-\frac{{\left(\mathit{\varepsilon} - 1\right)} \cos\left(v\right) - 1}{\mathit{\varepsilon} - 1}, \frac{u - 2}{\mathit{\varepsilon} - 1}, -\frac{{\left(\mathit{\varepsilon} - 1\right)} \sin\left(v\right) + 1}{\mathit{\varepsilon} - 1}\right) $ & $-r$ & $r$ & $\frac{1}{2} \pi$ & $\pi$ \\
106 & $ \left(-\frac{{\left(\mathit{\varepsilon} - 1\right)} \cos\left(u\right) \cos\left(v\right) - 1}{\mathit{\varepsilon} - 1}, -\frac{{\left(\mathit{\varepsilon} - 1\right)} \cos\left(u\right) \sin\left(v\right) - \sqrt{3} + 4}{\mathit{\varepsilon} - 1}, -\frac{{\left(\mathit{\varepsilon} - 1\right)} \sin\left(u\right) + 1}{\mathit{\varepsilon} - 1}\right) $ & $0$ & $\frac{1}{2} \pi$ & $\pi$ & $\frac{4}{3} \pi$ \\
107 & $ \left(-\frac{{\left(\mathit{\varepsilon} - 1\right)} \cos\left(v\right) - 1}{\mathit{\varepsilon} - 1}, -\frac{{\left({\left(\mathit{\varepsilon} - 1\right)} \sin\left(v\right) - r\right)} \cos\left(u\right) + 2}{\mathit{\varepsilon} - 1}, -\frac{{\left({\left(\mathit{\varepsilon} - 1\right)} \sin\left(v\right) - r\right)} \sin\left(u\right) + 1}{\mathit{\varepsilon} - 1}\right) $ & $\pi$ & $2 \pi$ & $\frac{2}{3} \pi$ & $\pi$ \\
108 & $ \left(\frac{\mathit{\varepsilon}}{\mathit{\varepsilon} - 1}, -\frac{r \cos\left(v\right) + 2}{\mathit{\varepsilon} - 1}, -\frac{r u \sin\left(v\right) + 1}{\mathit{\varepsilon} - 1}\right) $ & $0$ & $1$ & $0$ & $\pi$ \\
109 & $ \left(-\frac{\mathit{\varepsilon}}{\mathit{\varepsilon} - 1}, -\frac{r \cos\left(v\right) + 2}{\mathit{\varepsilon} - 1}, -\frac{r u \sin\left(v\right) + 1}{\mathit{\varepsilon} - 1}\right) $ & $0$ & $1$ & $0$ & $\pi$ \\
\bottomrule
\end{longtable}}